\documentclass[a4paper]{article}
\usepackage{a4,amssymb,latexsym,amsthm,paralist,dsfont,times}
\usepackage{amsmath}
\usepackage{color}
\usepackage[T1]{fontenc}
\usepackage[all]{xy}
\entrymodifiers={+!!<0pt,\fontdimen22\textfont2>}
\usepackage{tikz}
\usetikzlibrary{matrix,arrows}
\usepackage{tikz-cd}

\newtheoremstyle{alexthm}
  {.8cm}
  {.8cm}
  {\sl }
  {}
  {\bf}
  {.}
  {.5em}
  {}

\theoremstyle{alexthm}

\newtheorem{theorem}{Theorem}[section]
\newtheorem{corollary}[theorem]{Corollary}
\newtheorem{proposition}[theorem]{Proposition}
\newtheorem{lemma}[theorem]{Lemma}

\newtheorem{definition}[theorem]{Definition}

\newtheoremstyle{alexdef}
  {.8cm}
  {.8cm}
  {\rm }
  {}
  {\bf}
  {.}
  {.5em}
  {}
\theoremstyle{alexdef}

\newtheorem{remark}[theorem]{Remark}

\newcommand{\ab}{\mathit{ab}}

 \newcommand{\Spec}{\mathit{Spec}}
 \newcommand{\Gal}{\operatorname{Gal}}

 \newcommand{\Mor}{\operatorname{Mor}}
 \newcommand{\Aut}{\operatorname{Aut}}
 \newcommand{\Pic}{\operatorname{Pic}}
 
 \newcommand{\Hom}{\operatorname{Hom}}
 \newcommand{\et}{\mathrm{et}}
 \newcommand{\Nis}{\mathrm{Nis}}
 \newcommand{\A}{{\mathds  A}}
 \newcommand{\cA}{\mathcal{A}}
  \renewcommand{\P}{{\mathds  P}}
 \newcommand{\Q}{{\mathds  Q}}
 \newcommand{\Z}{{\mathds  Z}}
\newcommand{\T}{\mathcal{T}}
\newcommand{\Tor}{\,\mathcal{PHS}}
\newcommand{\Cor}{\operatorname{Cor}}
\newcommand{\lang}{\longrightarrow}
\newcommand{\ch}{\mathrm{char}}
\newcommand{\Sch}{\mathrm{Sch}}
\newcommand{\Sm}{\mathrm{Sm}}
\newcommand{\Sym}{\operatorname{Sym}}
\newcommand{\qfh}{\mathrm{qfh}}
\newcommand{\Ext}{\operatorname{Ext}}

\newcommand{\sing}{\mathrm{sing}}
\newcommand{\rec}{\mathit{rec}}
\newcommand{\U}{\mathcal{U}}
\newcommand{\V}{\mathcal{V}}

\newcommand{\id}{\mathrm{id}}

\newcommand{\ds}{\displaystyle}

\newcommand{\DM}{{\mathrm{DM}^{\mathrm{eff,-}}_\Nis(k,R)}}
\newcommand{\para}{\mathit{par}}
\newcommand{\taut}{\mathit{taut}}

\makeatletter
\let\@fnsymbol\@arabic
\makeatother
\title{\bf Tame Class Field Theory for Singular Varieties over Algebraically Closed Fields}
\author{by Thomas Geisser\footnote{Supported by JSPS Grant-in-Aid (B) 23340004}\;  and Alexander Schmidt\footnote{Supported by DFG-Forschergruppe FOR 1920}}
\begin{document}
\maketitle
\section{Introduction}
Let $X$ be a (possibly singular) separated scheme of finite type over an algebraically closed field $k$ of characteristic $p\ge 0$ and let $m$ be a natural number.
We construct a pairing between the first mod~$m$ algebraic singular homology
$H_1^S(X,\Z/m\Z)$ and the first  mod~$m$ tame \'{e}tale cohomology group
$H^1_{t}(X,\Z/m\Z)$. For  $\pi_1^{t,\ab}(X)=H^1_{t}(X,\Q/\Z)^\vee$ we prove the
following analogue of Hurewicz's theorem
in algebraic topology:

\begin{theorem}\label{main}
The induced homomorphism
\[
\rec_X: H_1^S(X,\Z/m\Z) \longrightarrow \pi_1^{t,\ab}(X)/m
\]
is surjective. It is an isomorphism of finite abelian groups if\/ $(m,p)=1$, and for general $m$ if resolution of singularities holds for schemes of dimension $\leq \dim X + 1$ over~$k$.
\end{theorem}

For $p\nmid m$,  the groups $H_1^S(X,\Z/m\Z)$ and $\pi_1^{t,\ab}(X)/m$ are known to be isomorphic by the work of Suslin and Voevodsky \cite{SV}. Theorem~\ref{main} above provides an explicit isomorphism which  extends to the case $p \mid m$ (under resolution of singularities). Moreover, in the last section we show that for $p\nmid m$ our isomorphism coincides with the one constructed in \cite{SV}.

\medskip
The motivation for constructing our pairing between the groups $H_1^S(X,\Z/m\Z)$ and $H^1_{t}(X,\Z/m\Z)$ comes from topology:
For  a locally contractible Hausdorff space $X$ and a natural number $m$, the canonical duality pairing
\[
\langle \cdot , \cdot \rangle: H_1^\sing (X,\Z/m\Z) \times H^1(X,\Z/m\Z) \longrightarrow \Z/m\Z,
\]
between singular homology and sheaf cohomology with mod $m$ coefficients can be given explicitly in the following way:
represent $b\in H^1(X,\Z/m\Z)$  by  a $\Z/m\Z$-torsor  $\mathcal{T}\to X$ and  $a\in H_1^\sing(X,\Z/m\Z)$ by a $1$-cycle $\alpha$ in the singular complex of $X$. Then
\[
\langle a, b\rangle = \Phi_{\para}^{-1} \circ \Phi_{\taut} \in \Z/m\Z,\ \text{ where }
\Phi_{\taut}, \Phi_{\para}\colon \alpha^*(\mathcal{T})|_0 \stackrel{\sim}{\to} \alpha^*(\mathcal{T})|_1
\]
are the isomorphisms between the fibres over $0$ and $1$ of the pull-back torsor $\alpha^*(\mathcal{T})\to \Delta^1=[0,1]$ given tautologically ($0^* \alpha=1^* \alpha$) and by parallel transport (every $\Z/m\Z$-torsor on $[0,1]$ is trivial).

For a variety $X$, the pairing between  $H_1^S(X,\Z/m\Z)$  and $H^1_t(X,\Z/m\Z)$ inducing the homomorphism $\rec_X$ of  our Main Theorem~\ref{main} will be constructed in the same way. However, $1$-cycles in the algebraic singular complex are not linear combinations of morphisms but finite correspondences from $\Delta^1$ to $X$. In order to mimic the above construction, we thus have to define the pull-back of a torsor along a finite correspondence,
which requires the construction of the push-forward torsor along a finite surjective morphism.

To prove Theorem~\ref{main}, we first consider the case of a
smooth curve $C$. If $\cA$ is the Albanese variety of $C$, then we have isomorphisms
\begin{equation} \label{neq1}
H_1^S(C,\Z/m\Z) \xrightarrow[\sim]{\;\delta \;} \null_m H_0^S(C,\Z) \cong \null_m \cA(k).
\end{equation}
The first isomorphism follows from the coefficient sequence together with the
divisibility of $H_1^S(C,\Z)$, and the second from the Abel-Jacobi theorem.
On the other hand,
\begin{equation} \label{neq2}
\Hom(\null_mA(k),\Z/m\Z) \xrightarrow[\sim]{\;\tau\;} H^1_t(C,\Z/m\Z).
\end{equation}
This follows because the maximal \'{e}tale subcovering $\widetilde \cA\to \cA$
of the $m$-mul\-ti\-pli\-cation map $\cA\stackrel{m}{\to} \cA$ is the quotient of $\cA$ by the connected component of the finite group scheme $\null_m \cA$, and
the maximal abelian tame \'{e}tale covering of $C$ with Galois group annihilated by $m$
is $\widetilde C := C\times_{\cA} \widetilde \cA$.
The heart of the proof of Theorem~\ref{main} for smooth curves is to show that under the above identifications,
our pairing agrees with the evaluation map.

We then show surjectivity of $\rec_X$ for general $X$ by reducing to the case of smooth curves. Finally, we
use duality theorems to show that both sides of $\rec_X$ have the same order:
For the $p$-primary part, we use resolution of singularities to reduce to the
smooth projective case considered  in \cite{ge-annals}.
For $(m,\ch(k))=1$,\linebreak[3] Suslin and Voevodsky \cite{SV} construct an isomorphism
\[
\alpha_X: H^1_\et(X,\Z/m\Z) \stackrel{\sim}{\to} H^1_S(X,\Z/m\Z).
\]
Hence the source and the target of $\rec_X$ have the same order and therefore $rec_X$ is an isomorphism.
In Section~\ref{appendix} we show that $\rec_X$ is dual to the map $\alpha_X$.
Thus, for $\ch(k)\nmid m$,  our construction gives an explicit description of the Suslin-Voevodsky isomorphism $\alpha_X$, which
zig-zags through Ext-groups in various categories and is difficult to understand.

\bigskip
The authors thank Takeshi Saito and Changlong Zhong for discussions during the early stages of the project.
It is a pleasure to thank Johannes Ansch\"{u}tz whose comments on an earlier version of this paper led to a substantial simplification of the proof of Theorem~\ref{vergleichepaarung}. Finally, we want to thank the referee for his helpful comments.

\section{Torsors and finite correspondences}
\label{torseur}

\medskip
All occurring schemes in this section are separated schemes of finite type over a field $k$.
For  any abelian group $A$ and a finite surjective morphism $\pi: Z \to X$ with $Z$ integral and $X$ normal, connected,  we have transfer maps  \[
\pi_*: H^i_\et(Z,A)\to H^i_\et(X,A)\]  for all $i\geq 0$ (see \cite{MVW}, 6.11, 6.21).  The group $H^1_\et(Z,A)$ classifies isomorphism classes of \'{e}tale $A$-torsors (i.e., principal homogeneous spaces) over the scheme $Z$.
We are going to construct a functor
\[
\pi_*: \Tor(Z,A) \lang \Tor(X,A)
\]
from the category of \'{e}tale $A$-torsors on $Z$ to the category of \'{e}tale $A$-torsors on $X$, which induces the transfer map $\pi_*: H^1_\et(Z,A)\to H^1_\et(X,A)$ above on isomorphism classes.

We recall how to add and subtract torsors.
For an abelian group $A$ and $A$-torsors $\T_1$, $\T_2$  on a scheme $Y$, define
\[
\T_1 + \T_2
\]
to be the quotient scheme of $\T_1 \times_Y \T_2$ by the action of $A$ given by $(t_1,t_2)+a= (t_1+a,t_2 - a)$. It carries the structure of an $A$-torsor by setting
\[
\overline{(t_1,t_2)}+a:=\overline{(t_1+a,t_2)} \quad (=\overline{(t_1,t_2+a)}).
\] The functor
\[
+ : \Tor(Y,A) \times \Tor(Y,A) \lang \Tor(Y,A), \quad (\T_1,\T_2)\longmapsto \T_1 + \T_2,
\]
lifts the addition in $H^1_\et(Y, A)$ to torsors (cf. \cite{Mi}, III, Rem.~4.8\,(b)).
Note that ``$+$'' is associative and commutative up to natural functor isomorphisms. In particular, we can multiply a torsor by any natural number $m$, putting $m\cdot \T= \T + \cdots + \T$ ($m$ times). If $mA=0$, then we have
a natural isomorphism of torsors
\begin{equation}\label{eq1}
m\cdot \T \ \stackrel{\sim}{\lang} \ Y \times A , \quad \overline{(t_1,\ldots,t_m)} \mapsto (t_2-t_1)+ \cdots + (t_m-t_1) \in A,
\end{equation}
where $Y \times A$ is the trivial $A$-torsor on $Y$ representing the constant sheaf $\underline{A}$ over $Y$. Here $t_i-t_j$ denotes the unique element $a\in A$ with $t_i=t_j +a$.

\medskip
Furthermore, given a torsor $\T$,  define $(-\T)$ to be the torsor which is isomorphic to $\T$ as a scheme and  on which $a\in A$ acts as $-a$. This yields a functor
\[
(-1):\  \Tor(Y,A) \lang \Tor(Y,A), \quad \T \longmapsto (-\T),
\]
which lifts multiplication by $(-1)$ from $H^1_\et(Y, A)$ to an endofunctor of $\Tor(Y,A)$. We have a natural isomorphism of torsors
\begin{equation}\label{eq2}
\T + (-\T) \stackrel{\sim}{\lang} Y \times A, \quad \overline{(t_1,t_2)} \mapsto t_1-t_2 \in A.
\end{equation}

\bigskip
Now let $\pi: Z \to X$ be finite and surjective, $Z$ integral, $X$ normal, connected, and let $\T$ be an $A$-torsor on $Z$.
For every point $x\in X$, the base change  $Z \times_X X_x^{sh}$ is a product of  strictly  henselian local schemes. Therefore we find an \'{e}tale cover $(U_i\to X)_{i\in I}$ of $X$ such that $\T$ trivializes over the pull-back \'{e}tale cover $(\pi^{-1}(U_i)\to Z)_{i\in I}$ of $Z$.

Next choose a pseudo-Galois covering $\widetilde \pi: \widetilde Z \to X$ dominating $Z\to X$. Recall that this means that $k(\widetilde Z)|k(X)$ is a normal field extension and that  the natural map $\Aut_{X}(\widetilde Z) \to \Aut_{k(X)}(k(\widetilde Z))$ is bijective (cf.\ \cite{SV}, Lemma 5.6). Let $\pi_{in}: X_{in} \to X$ be the quotient scheme $\widetilde Z /G$, where $G=\Aut_X(\widetilde Z)$. Then $X_{in}$ is the normalization of $X$ in the maximal purely inseparable subextension $k(X)^{in} /k(X)$ of $k(\widetilde Z)/k(X)$. Consider the object
\[
\widetilde \T:= \sum_{\varphi \in \Mor_{X}(\widetilde Z, Z)} \varphi^*(\T) \ \in \ \Tor(\widetilde Z,A),
\]
which is defined up to unique isomorphism. Starting from any trivialization of $\T$ over $(\pi^{-1}(U_i)\to Z)_{i\in I}$, we obtain a trivialization of the restriction of $\widetilde \T$ to $(\widetilde \pi^{-1}(U_i)\to \widetilde Z)_{i\in I}$ of the form
\[
\widetilde \T|_{\widetilde \pi^{-1}(U_i)} \cong \widetilde \pi^{-1}(U_i) \times A,
\]
where $G=\Aut_X(\widetilde Z)$ acts on the right  hand side in the canonical way on $\widetilde \pi^{-1}(U_i)$ and trivially on $A$.
Therefore  the quotient scheme $\widetilde \T/G$ is an $A$-torsor on $\widetilde Z/G=X_{in}$ in a natural way. Since $X_{in} \to X$ is a topological isomorphism, $\widetilde \T/G$  comes by base change from a unique $A$-torsor $\T'$ on $X$.
\begin{definition}
The push-forward $A$-torsor $\pi_*(\T)$ on $X$ is defined by
\[
\pi_*(\T)= [k(Z):k(X)]_{in} \cdot \T'.
\]
\end{definition}
The assignment $\T \mapsto \pi_*(\T)$ defines a functor
\[
\pi_*: \Tor(Z,A)\lang \Tor(X,A).
\]
The functor $\pi_*$ is additive in the sense that it commutes with the functors ``$+$'' and ``$(-1)$'' up to a natural functor isomorphism.

\bigskip
Let $\T \in \Tor(Z,A)$ and assume that there exists a section $s: Z \to \T$ to the projection $\T \to Z$ (so $\T$ is trivial and $s$ gives a trivialization). Let again  $\pi: Z \to X$ be finite and surjective, $Z$ integral, $X$ normal, connected. Then
\[
\widetilde \T:= \sum_{\varphi \in \Mor_{X}(\widetilde Z, Z)} \varphi^*(\T) \ \in \ \Tor(\widetilde Z,A)
\]
has the canonical section $\sum_{\varphi \in \Mor_{X}(\widetilde Z, Z)} \varphi^*(s)$ over $\widetilde Z$. It descends to a section of $\T/G$ over $\widetilde Z /G = X_{in}$. Descending to $X$ and multiplying by $[k(X_{in}):k(X)]$, we obtain a section
\[
\pi_*(s): X \to \pi_*(\T).
\]
In other words, we obtain a map
\[
\pi_*: \Gamma (Z, \T) \lang \Gamma(X,\pi_*(\T));
\]
hence every trivialization of $\T$ gives a trivialization of $\pi_*(\T)$ in a natural way.

\bigskip
In order to see that $\pi_*$ induces the transfer map $\pi_*: H^1_\et(Z,A) \to H^1_\et(X,A)$ after passing to isomorphism classes, we formulate the construction of $\pi_*$ on the level of  \v Cech $1$-cocycles. As explained above, we find an \'{e}tale cover $(U_i\to X)_{i\in I}$ such that $\T$ trivializes over the \'{e}tale cover $(\pi^{-1}(U_i) \to Z)_{i\in I}$ of $Z$. We fix a trivialization and obtain a \v Cech $1$-cocycle
\[
a=(a_{ij}\in \Gamma(\pi^{-1}(U_i\times_XU_j),A))
\]
over  $(\pi^{-1}(U_i)\to Z)_{i\in I}$ which defines $\T$.
As before choose a pseudo-Galois covering $\widetilde \pi: \widetilde Z \to X$ dominating $Z\to X$.  Now for all $i,j$ consider the element
\[
\sum_{\varphi \in \Mor_{X}(\widetilde Z, Z)} \varphi^*(a_{ij}) \in \Gamma(\widetilde \pi^{-1}(U_i \times_X U_j),A)
\]
which, by Galois invariance, lies in
\[
\Gamma(\pi_{in}^{-1}(U_i\times_XU_j), A) = \Gamma(U_i\times_XU_j, A).
\]
The  \v Cech $1$-cocycle given by
\[
[k(Z):k(X)]_{in} \cdot \Big(\sum_{\varphi \in \Mor_{X}(\widetilde Z, Z)} \varphi^*(a_{ij})\Big) \in \Gamma(U_i\times_XU_j, A).
\]
now defines a trivialization of $\pi_*(\T)$ over $(U_i\to X)_{i\in I}$. Since the transfer map on \'{e}tale cohomology is defined on \v Cech cocycles in exactly this way (see \cite{MVW}, 6.11, 6.21), we obtain

\begin{lemma} Passing to isomorphism classes, the functor $\pi_*: \Tor(Z,A) \to \Tor(X,A)$ constructed above induces the transfer homomorphism
\[
\pi_*: H^1_\et(Z,A) \to H^1_\et (X,A).
\]
\end{lemma}

If any finite subset of closed points of $X$ is contained in an affine open, then symmetric powers exist, and another description of the push-forward for torsors is the following:
Associated with the finite morphism $\pi: Z \to X$ of degree $d$, there is a section $s_\pi: X \to \Sym^d (Z/X)$ to the natural projection $\Sym^d (Z/X) \to X$ (see (\cite{SV}, p.~81). We denote the composite of $s_\pi$ with $pr: \Sym^d(Z/X) \to \Sym^d(Z)$ by $S_\pi$. Defining  $f: \widetilde Z \to \Sym^d(Z/X)$ by repeating each element in $\Mor_{X}(\widetilde Z, Z)$  exactly $[k(Z):k(X)]_{in}$-times,  the diagram
\[
\begin{tikzcd}
\widetilde Z  \arrow{rr}{f}\arrow{rrdd}[swap]{\widetilde{\pi}} && \Sym^d(Z/X)\arrow{r}{pr}& \Sym^d(Z)\\
\\
&&X\arrow{uu}{s_\pi}\arrow{ruu}[swap]{S_\pi}
\end{tikzcd}
\]
commutes. For  an $A$-torsor $\T\to Z$,  the $d$-fold self-product $ \T \times_k  \cdots \times_k \T$ is an $A^{d}$-torsor over the $d$-fold self-product of $Z$ in a natural way.
Taking the quotient by the $A^{d-1}$-action
\[
(a_1,\ldots,a_{d-1}) (t_1,\ldots, t_d) = (t_1 +a_1, t_2-a_1+ a_2,t_3-a_2+a_3,\ldots, t_d -a_{d-1}),
\]
we obtain an $A$-torsor over $Z^d$.  Dividing out the by the action of the symmetric group $S_d$, we obtain an $A$-torsor over  $\Sym^d(Z)$ and denote it by $\Sym^d_A(\T)$.
We obtain natural isomorphisms in $\Tor(\widetilde Z,A)$:
\[
\begin{array}{rcl}
\ds [k(Z):k(X)]_{in} \cdot \!\!\!\!\!\sum_{\varphi \in \Mor_{X}(\widetilde Z, Z)} \!\!\!\!\! \varphi^*(\T)& \cong & (pr\circ f)^* \Sym^d_A(\T)\\
 & \cong & \widetilde \pi^* \circ (pr\circ s_\pi)^*\Sym^d_A(\T).
\end{array}
\]  By our construction of $\pi_*(\T)$ we obtain
\begin{lemma} \label{symprod} We have a natural isomorphism in $\Tor(X,A)$:
\[
\pi_*(\T) \cong S_\pi^*(\Sym^d(\T)),
\]
where $S_\pi= pr\circ s_\pi: X \to \Sym^d(Z)$.
\end{lemma}

\bigskip

Assume now that $X$ is regular and $Y$ arbitrary. The group of finite correspondences $\Cor(X,Y)$ is defined as the free abelian group on the set of integral subschemes $Z\subset X \times Y$ which project finitely and surjectively to a connected component of~$X$.
For such a $Z$, we define $p_{[Z\to X]\, *}:  \Tor(Z,A) \to \Tor(X,A)$ by extending (if $X$ is not connected) the push-forward torsor defined above in a trivial way to those connected components of $X$ which are not dominated by $Z$.
We consider the functor
\[
[Z]^*= p_{[Z\to X]\, *} \circ p_{[Z\to Y]}^* : \Tor (Y, A) \lang \Tor(X,A).
\]
Using the operations ``$+$'' and ``$(-1)$''  we extend this construction to arbitrary finite correspondences.

\begin{definition}
Let $X$ be regular, $Y$ arbitrary and $\alpha= \sum n_i Z_i \in \Cor(X,Y)$ a finite correspondence. Then
\[
\alpha^*: \Tor(Y,A) \lang \Tor(X,A)
\]
is defined by  setting
\[
 \alpha^*(\T):= \sum n_i [Z_i]^*(\T).
\]
\end{definition}

Using the isomorphism (\ref{eq2}) above, we immediately obtain
\begin{lemma}
For $\alpha_1,\alpha_2\in Cor(X,Y)$ and $\T_1, \T_2 \in \Tor(Y,A)$, $n_1,n_2\in \Z$, we have   a natural isomorphism
\[
(\alpha_1+\alpha_2)^*(n_1\T_1 + n_2 \T_2) \cong n_1\alpha_1^*(\T_1) +  n_1\alpha_2^*(\T_1) + n_2\alpha_1^*(\T_2) + n_2\alpha_2^*(\T_2).
\]
\end{lemma}

\bigskip
If $X$ and $Y$ are regular and $Z$ is arbitrary, we have a natural composition law
\[
\Cor(X,Y) \times \Cor(Y,Z) \lang \Cor(X,Z),\ (\alpha,\beta) \mapsto \beta \circ \alpha,
\]
(see \cite{MVW}, Lecture 1). A straightforward but lengthy computation unfolding the definitions shows
\begin{proposition} \label{compo}
Let $X$ and $Y$ be regular and $Z$ arbitrary.  Let $\alpha\in \Cor(X,Y)$ and $\beta\in \Cor(Y,Z)$.  Then, for any $\T\in \Tor(Z,A)$, we have  a canonical isomorphism
\[
\alpha^*(\beta^*(\T)) \cong (\beta\circ \alpha)^*(\T).
\]
\end{proposition}

Finally, assume that $mA=0$ for some natural number $m$. Then (using the isomorphism (\ref{eq1}) above), we have for any $\alpha,\beta\in Cor(X,Y)$, $\T \in \Tor(Y,A)$, a natural isomorphism
\[
(\alpha+m\beta)^*(\T) \cong \alpha^*(\T).
\]
Therefore, we have an $A$-torsor
\[
\bar \alpha^*(\T) \in \Tor(X,A)
\]
given up to unique isomorphism for any $\bar \alpha\in \Cor(X,Y)\otimes \Z/m\Z$. In other words, we obtain the

\begin{lemma}
Assume that $m A=0$, and let $\alpha,\beta\in \Cor(X,Y)$ have the same image in $\Cor(X,Y)\otimes \Z/m\Z$. Then there is a natural isomorphism of functors
\[
\alpha^* \cong \beta^* : \Tor(Y,A) \to \Tor(X,A).
\]
\end{lemma}

For a regular connected curve $C$ we consider the subgroup $H^1_t(C,A)\subseteq H^1_\et(C,A)$ of tame cohomology classes (corresponding to those continuous homomorphisms $\pi_1^\et(C)\to A$ which factor through the tame fundamental group $\pi_1^t(\bar C, \bar C -C)$, where $\bar C$ is the unique regular compactification of $C$).

For a general scheme $X$ over $k$ we call a cohomology class in $a\in H^1_\et(X,A)$ curve-tame (or just tame) if for any morphism $f:C\to X$ with $C$ a regular curve, we have $f^*(a)\in H^1_t(C,A)$. The tame cohomology classes form a subgroup
\[
H^1_t(X,A)\subseteq  H^1_\et(X,A).
\]
The groups coincide if $X$ is proper or if $p=0$ or if $p>0$ and
$A$ is $p$-torsion free, where $p$ is the characteristic of the base field $k$.

\begin{definition}
We call an \'{e}tale $A$-torsor $\T$ on $X$ \emph{tame} if its isomorphism class lies in $H^1_t(X,A)\subseteq H^1_\et(X,A)$.
\end{definition}

\begin{lemma}\label{tametotame}
Let $Z$ be integral, $X$  normal, connected, $\pi:Z\to X$ finite, surjective and $f: Z\to Y$ any morphism. Let $\T$ be a tame torsor on $Y$. Then $\pi_*(f^*(\T))$ is a tame torsor on $X$.
\end{lemma}

\begin{proof} By definition, $f^*$ preserves curve-tameness. So we may assume $Z=Y$, $f=\id$. Again by the definition of curve-tameness and using Proposition~\ref{compo}, we may reduce to the case that $X$ is a regular curve.  Since \'{e}tale cohomology commutes with direct limits of coefficients, we may assume that $A$ is a finitely generated abelian group. Furthermore, we may assume that $\ch(k)=p>0$ and $A=\Z/p^r\Z$, $r\geq 1$.

Let $\bar Z$ be the canonical compactification of $Z$, i.e., the unique proper curve over~$k$ which contains $Z$ as a dense open subscheme and such that all points of $\bar Z \smallsetminus Z$ are regular points of $\bar Z$. By the definition of tame coverings of curves, $\T$ extends to a $\Z/p^r\Z$-torsor on $\bar Z$. Hence also $\pi_*(\T)$ extends to the canonical compactification $\bar X$ of $X$ and so is tame.
\end{proof}

\begin{proposition} \label{numerical}
Let $\bar X$ be a proper and regular scheme over $k$ and let $X\subset \bar X $ be a dense open subscheme. Let $p=\ch(k) >0$.  Then for any $r\geq 1$ the natural inclusion
\[
 H^1_\et(\bar X,\Z/p^r\Z) \hookrightarrow  H^1_\et(X,\Z/p^r\Z)
\]
induces an isomorphism
\[
H^1_\et(\bar X,\Z/p^r\Z)= H^1_t(\bar X,\Z/p^r\Z) \stackrel{\sim}{\longrightarrow} H^1_t(X,\Z/p^r\Z) \subseteq H^1_\et(X,\Z/p^r\Z).
\]

\end{proposition}

\begin{proof}
Let $\T_0$ be any connected component of a tame $\Z/p^r\Z$-torsor $\T$ on $X$. Then the morphism $\T_0\to X$ is curve-tame in the sense of \cite{KS}, \S4, and $\T_0$ is the normalization of $X$ in the abelian field extension of $p$-power degree $k(\T_0)/k(X)$.  By \cite{KS}, Thm.\ 5.4.\,(b), $\T_0\to X$ is numerically tamely ramified along $\bar X \smallsetminus X$. This means that the inertia groups in $\Gal(k(\T_0)/k(X))$ of all points $\bar x \in \bar X \smallsetminus X$ are of order prime to $p$, hence trivial. Therefore $\T_0$, and thus $\T$ extends to $\bar X$.
\end{proof}

\begin{corollary} \label{simplexex}
Let $\Delta^n=\Spec(k[T_0,\ldots,T_n]/\sum T_i=1)$ be the $n$-dimensional standard simplex over $k$ and let $A$ be an abelian group. Then
\[
H^1_t(\Delta^n,A)\cong H^1_\et(k,A).
\]
In particular, $H^1_t(\Delta^n,A)=0$ if $k$ is separably closed.
\end{corollary}

\begin{proof}  Since tame cohomology commutes with direct limits of coefficients, and since $H^1_\et(\Delta^n,\Z)=0$,   we may assume that $A\cong \Z/m\Z$ for some $m\geq 1$.  If $p\nmid m$, we obtain:
\[
H^1_t(\Delta^n,\Z/m\Z)\cong H^1_\et(\A^n,\Z/m\Z)\cong H^1_\et(k,\Z/m\Z).
\]
If $p=\ch(k)>0$ and $m=p^r$, $r\geq 1$, Proposition~\ref{numerical} yields
\[
H^1_t(\Delta^n,\Z/p^r\Z) \cong H^1_t(\A^n,\Z/p^r\Z)\stackrel{\sim}{\leftarrow}H^1_t(\P^n,\Z/p^r\Z)= H^1_\et(\P^n,\Z/p^r\Z).
\]
Finally note that $H^1_\et(\P^n,\Z/p^r\Z)\cong H^1_\et(k,\Z/p^r\Z)$.
\end{proof}

In the following, let  $k$ be an algebraically closed field of characteristic $p\geq 0$ and let $X$ be a separated scheme of finite type over $k$. Let $H_i^S(X,\Z/m\Z)$ denote the mod-$m$ Suslin homology, i.e., the $i$-th homology group of the complex
\[
\Cor(\Delta^\bullet, X) \otimes \Z/m\Z.
\]
Let $A$ be an abelian group with $mA=0$. We are going to construct a pairing
\[
H_1^S(X,\Z/m\Z) \times H^1_t(X,A) \longrightarrow A
\]
as follows: let $\T\to X$ be a tame $A$-torsor representing a class in $H^1_t(X,A)$ and let $\alpha\in \Cor(\Delta^1,X)$ be a finite correspondence representing a $1$-cocycle in the mod-$m$ Suslin complex. Then
\[
\alpha^*(\T)
\]
is a torsor over $\Delta^1$. Since $\alpha$ is a cocycle modulo $m$, $(0^*- 1^*)(\alpha)$ is of the form $m \cdot z$ for some $z\in \Cor(\Delta^0,X)=\Z^{(X(k))}$. We therefore obtain a  canonical identification
\[
\Phi_{\taut}: 0^*(\alpha^*(\T)) \stackrel{\sim}{\lang} 1^*(\alpha^*(\T))
\]
of $A$-torsors over $\Delta^0=\Spec(k)$. Furthermore, by Corollary~\ref{simplexex}, the tame torsor $\alpha^*(\T)$ on $\Delta^1$ is trivial, hence a disjoint union of copies of $\Delta^1$. By parallel transport, we obtain another identification
\[
\Phi_{\para}: 0^*(\alpha^*(\T))\stackrel{\sim}{\lang}1^*(\alpha^*(\T)).
\]
Hence there is a unique $\gamma(\alpha, \T)\in A$ such that
\[
 \Phi_{\para}= \text{(translation by $\gamma(\alpha, \T)$)} \circ \Phi_{\taut}.
\]

\begin{proposition}\label{pairing}
The element  $\gamma(\alpha, \T)\in  A$ only depends on the class of\/ $\T$ in $H^1_t(X,A)$ and on the class of $\alpha$ in $H^S_1(X,\Z/m\Z)$. We obtain a bilinear pairing
\[
\langle \cdot , \cdot \rangle :  H^S_1(X,\Z/m\Z) \times H^1_t(X,A)   \longrightarrow A.
\]
\end{proposition}

\begin{proof}
Replacing $\T$ by another torsor isomorphic to $\T$ does not change anything. The nontrivial statement is that $\langle \alpha, \T\rangle$ only depends on the class of $\alpha$ in $H^S_1(X,\Z/m\Z)$. For $\beta\in \Cor(\Delta^1,X)$, we have
\[
\langle \alpha + m\beta, \T \rangle = \langle \alpha, \T\rangle + m \langle \beta,\T\rangle = \langle\alpha, \T \rangle.
\]
It therefore remains to show that
\[
\langle \partial^*(\Phi),\T \rangle=0,
\]
for all $\Phi \in \Cor(\Delta^2,X)$, where $\partial_i: \Delta^1 \to \Delta^2 $, $i=0,1,2$, are the face maps and $\partial^*(\Phi)= \Phi\circ \partial_0 -  \Phi \circ\partial_1 + \Phi\circ\partial_2$. Considering $\partial=\partial_0-\partial_1+\partial_2$ as a finite correspondence from $\Delta^1$ to $\Delta^2$, it represents a cocycle in the singular complex $\Cor(\Delta^\bullet,\Delta^2)$. Proposition~\ref{compo} implies that
\[
\langle \partial^*(\Phi),\T \rangle=\langle \Phi\circ \partial,\T \rangle= \langle \partial, \Phi^*(\T) \rangle.
\]
By Corollary~\ref{simplexex}, the tame torsor $\Phi^*(\T)$ is trivial on $\Delta^2$. Hence $\langle \partial, \Phi^*(\T) \rangle=0$.
\end{proof}

\bigskip
In the following, we use the notation $\pi_1^{t,\ab}(X):=H^1_{t}(X,\Q/\Z)^*$. If $X$ is connected, then $\pi_1^{t,\ab}(X)$ is the abelianized (curve-)tame fundamental group of $X$, see \cite{KS}, \S 4.

\begin{definition}
For $m\ge 1$ we define
\[
\rec_X: H_1^S(X,\Z/m\Z) \lang \pi_1^{t,\ab}(X)/m
\]
as the homomorphism induced by the pairing of Proposition \ref{pairing} for $A=\Z/m\Z$  combined with the isomorphism $H^1_t(X,\Z/m\Z)^*\cong \pi_1^{t,\ab}(X)/m$.
\end{definition}

The statement of the next lemma immediately follows from the definition of $\rec$.

\begin{lemma}
Let $f: X' \to X$ be a morphism of separated schemes of finite type over $k$. Then the induced diagram
\[
\begin{tikzpicture}
\matrix (m) [matrix of math nodes, row sep=2em,
column sep=2em, text height=2.5ex, text depth=1.25ex]
{
H_1^S(X',\Z/m\Z) & \pi_1^{t,\ab}(X')/m\\
H_1^S(X,\Z/m\Z) & \pi_1^{t,\ab}(X)/m\\
};
\path[->,font=\scriptsize] (m-1-1) edge node[auto]{$\rec_{X'}$} (m-1-2);
\path[->,font=\scriptsize] (m-1-1) edge node[auto]{$f_*$} (m-2-1);
\path[->,font=\scriptsize] (m-1-2) edge node[auto]{$f_*$} (m-2-2);
\path[->,font=\scriptsize] (m-2-1) edge node[auto]{$\rec_X$}(m-2-2);
\end{tikzpicture}
\]
commutes.
\end{lemma}

\section{Rigid \v{C}ech complexes} \label{chech-sect}

We consider \'{e}tale sheaves $F$ on the category $\Sch/k$ of separated schemes of finite type over a field  $k$. By a result of M.\ Artin, \v{C}ech cohomology $\check H^\bullet(X,F)$ and sheaf cohomology $H^\bullet_\et(X,F)$ coincide in degree $\leq 1$ and in arbitrary degree if $X$ is quasi-projective (cf.\ \cite{Mi}, III Thm.\ 2.17).
Comparing the \v{C}ech complex  for a covering $\U$ and that for a finer covering $\V$, the refinement homomorphism
\[
\check C^\bullet (\U, F) \lang \check C^\bullet (\V, F)
\]
is canonical only up to chain homotopy and hence only the induced map $\check H^\bullet (\U, F)\allowbreak \to \check H^\bullet (\V, F)$ is well-defined.
    We can remedy this problem in the spirit of Friedlander \cite{Fr}, chap.4, by using rigid coverings:

We fix an algebraic closure  $\bar k /k$.  A \emph{rigid \'{e}tale covering} $\mathcal{U}$ of $X$ is a family of pointed  separated \'{e}tale morphisms
\[
(U_x, u_x) \lang (X, x), \quad x \in X(\bar k),
\]
with $U_x$ connected and  $u_x \in U_x(\bar k)$ mapping to $x$. For an \'{e}tale sheaf $F$ the rigid  \v{C}ech complex is defined by
\[
\check C^\bullet(\U, F): \quad \check C^n(\U, F)= \prod_{(x_0,\ldots,x_n)\in X(\bar k)^{n+1}} \Gamma (U_{x_0}\times_X \cdots \times_X U_{x_n}, F)
\]
with the usual differentials. It is clear what it means for a rigid covering  $\mathcal V$ to be a  refinement of $\mathcal U$. Because the marked points map to each other, there is exactly one refinement morphism, hence we obtain a canonical refinement morphism on the level of complexes
\[
\check C^\bullet(\U, F) \to \check C^\bullet(\V, F).
\]
The set of rigid coverings is cofiltered (form the fibre product for each $x\in X(\bar k)$ and restrict to the connected components of the marked points). Therefore we can define the rigid \v{C}ech complex of $X$ with values in $F$ as the filtered direct limit
\[
\check C^\bullet(X, F) := \varinjlim_\U \check C^\bullet(\U, F),
\]
where $\U$ runs through all rigid coverings of $X$. Note that the rigid \v Cech complex depends on the structure morphism $X\to k$ and not merely on the scheme $X$.

Forgetting the marking, we can view a rigid covering as a usual covering. Every covering can be refined by a covering which arises by forgetting the marking of a rigid covering. Hence the cohomology of the rigid \v{C}ech complex coincides with the usual \v{C}ech cohomology of $X$ with values in $F$.

\medskip
For a morphism $f: Y \to X$ and a rigid  \v{C}ech covering $\U/X$, we obtain a rigid  \v{C}ech covering  $f^*\U/Y$ by taking base extension to $Y$ and restricting to the connected components of the marked points, and in the limit we obtain a homomorphism
\[
f^*: \check C^\bullet(X, F) \lang \check C^\bullet(Y,F).
\]

\begin{lemma}\label{quasifinite}
If\/ $\pi: Y\to X$ is quasi-finite, then the rigid coverings of the form $\pi^*\U$ are cofinal among the rigid coverings of\/ $Y$.
\end{lemma}

\begin{proof}
This is an immediate consequence of the fact that a quasi-finite and separated scheme $Y$ over the spectrum $X$ of a henselian ring is of the form $Y= Y_0 \sqcup Y_1 \sqcup \ldots \sqcup Y_r$ with $Y_0\to X$ not surjective and $Y_i\to X$ finite surjective with $Y_i$ the spectrum of a henselian ring, $i=1,\ldots,r$, cf.\ \cite{Mi}, I,\ Thm.\ 4.2.
\end{proof}

\begin{lemma}
If $F$ is $\qfh$-sheaf on  $\Sch/k$, then for any $n\geq 0$ the presheaf $\underline{\check C}^n(- , F)$ given by
\[
X \longmapsto \check C^n(X,F)
\]
is a $\qfh$-sheaf. The obvious sequence
\[
0 \to F \to \underline{\check C}^0(- , F) \to \underline{\check C}^1(- , F) \to \underline{\check C}^2(- , F) \to \cdots
\]
is exact as a sequence of \'{e}tale (and hence also of $\qfh$) sheaves.
\end{lemma}

\begin{proof}
We show that each $\underline{\check C}^n(- , F)$ is a $\qfh$-sheaf. For this, let $\pi: Y \to X$ be a $\qfh$-covering, i.e., a quasi-finite  universal topological epimorphism.
We denote the projection by $\Pi: Y\times_X Y \to X$. By Lemma~\ref{quasifinite}, we have to show that the sequence
\[
\varinjlim_\U \check C^n(\U, F) \to \varinjlim_\U \check C^n(\pi^*\U, F) \rightrightarrows \varinjlim_\U \check C^n(\Pi^*\U, F )
\]
is an equalizer, where $\U$ runs through the  rigid coverings of $X$. Since filtered colimits commute with finite limits, it suffices to show the exactness for a single, sufficiently small $\U$.  This, however, follows from the assumption that $F$ is a $\qfh$-sheaf.

Finally, the exactness of\/ $0 \to F \to \underline{\check C}^0(- , F) \to \underline{\check C}^1(- , F) \to  \cdots $ as a sequence of \'{e}tale sheaves follows by considering stalks.
\end{proof}

\noindent
Being $\qfh$-sheaves, the sheaves $F$ and $\underline{\check C}^n(- , F)$ admit transfer maps, see \cite{SV},\,\S5. For later use, we make the relation between the transfers of $F$ and of $\underline{\check C}^n(- , F)$ explicit:
Let $Z$ be integral, $X$ regular and $\pi: Z\to X$ finite and surjective. Let $F$ be a $\qfh$-sheaf on $\Sch/k$.
For $x\in X(\bar k)$ we have
\[
X_x^{sh}\times_X Z = \coprod_{z\in \pi^{-1}(x)} Z_z^{sh},
\]
where $\pi^{-1}(x)$ denotes the set of morphisms $z: \Spec(\bar k) \to Z$ with $\pi\circ z= x$. For sufficiently small \'{e}tale $(U_x,u_x) \to (X,x)$, the set of connected components of $U_x \times_X Z$ is in 1-1-correspondence with the set $\pi^{-1}(x)$, and to each family of \'{e}tale morphisms
\[
(V_z,v_z) \lang (Z,z), \quad z \in \pi^{-1}(x),
\]
there is (after possibly making $U_x$ smaller) a unique morphism
\[
U_x \times_X Z \lang \coprod_{z\in \pi^{-1}(x)} V_z,
\]
over $Z$, which sends the connected component associated with $z$ of $U_x \times_X Z$ to $V_z$, and the point $(u_x,z)$ to $v_z$.

\medskip
In this way we obtain, for finitely many points $(x_0,\ldots,x_n)$, $n\geq 0$, and for every family
\[
(V_{z_i,v_{z_i}}) \lang (Z,z_i), \quad z_i \in \pi^{-1}(x_i),
\]
and sufficiently small chosen
\[
(U_{x_i},u_{x_i}) \lang (X,x_i),\quad i=0,\ldots,n,
\]
a homomorphism
\[
\prod_{\substack{(z_0,\ldots,z_n)\\z_i\in \pi^{-1}(x_i)}}\Gamma\big( V_{z_0} \times_Z \cdots \times_Z V_{z_n},F\big) \lang \Gamma\big(U_{x_0} \times_X \cdots \times_X U_{x_n} \times_X Z, F\big).
\]
Since $F$ is a $\qfh$-sheaf, we can compose this with the transfer map associated with the finite morphism
\[
U_{x_0} \times_X \cdots \times_X U_{x_n} \times_X Z \to U_{x_0} \times_X \cdots \times_X U_{x_n}.
\]
Forming for fixed $n$ the product over all  $(x_0,\ldots,x_n)\in X(\bar k)^{n+1}$ and passing to the limit over all rigid coverings, we obtain the transfer homomorphism
\[
\pi_*: \check C^\bullet(Z,F) \lang \check C^\bullet (X,F).
\]
Passing to cohomology, we obtain the usual transfer on \'{e}tale cohomology in degree $0$ and $1$, and in any degree if the schemes are quasi-projective.

\bigskip

Next we give the pairing
\[
\langle \cdot , \cdot \rangle :  H^S_1(X,\Z/m\Z) \times H^1_t(X,A)   \longrightarrow A.
\]
constructed in Proposition~\ref{pairing} for $k$ algebraically closed and an abelian group $A$ with $mA=0$ the following interpretation in terms of the rigid \v{C}ech complex:

\medskip\noindent
Let $a\in H_1^S(X,\Z/m\Z)$ and $b\in H^1_t(X,A)$ be given, and let $\alpha\in \Cor_k(\Delta^1,X)$ and $\beta \in \ker(\check C^1(X,A) \stackrel{d}{\to} \check C^2(X,A))$ be representing elements. Note that $(0^*-1^*)(\alpha)\in m \Cor(\Delta^0,X)$ by assumption.  Consider the diagram

\[
\begin{tikzpicture}
\matrix (m) [matrix of math nodes, row sep=2em,
column sep=2em, text height=2.5ex, text depth=1.25ex]
{& \check C^0(X,A) &  \check C^1(X,A) & \check C^2(X,A) \\
& \check C^0(\Delta^1,A) &  \check C^1(\Delta^1,A) & \check C^2(\Delta^1,A) \\
A& \check C^0(\Delta^0,A) &  \check C^1(\Delta^0,A) & \check C^2(\Delta^0,A) \\
};
\path[->,font=\scriptsize] (m-1-2) edge node[auto]{$d$} (m-1-3);
\path[->,font=\scriptsize] (m-1-3) edge node[auto]{$d$} (m-1-4);
\path[->,font=\scriptsize] (m-2-2) edge node[auto]{$d$} (m-2-3);
\path[->,font=\scriptsize] (m-2-3) edge node[auto]{$d$} (m-2-4);
\path[right hook->,font=\scriptsize] (m-3-1) edge node[auto]{} (m-3-2);
\path[->,font=\scriptsize] (m-3-2) edge node[auto]{$d$} (m-3-3);
\path[->,font=\scriptsize] (m-3-3) edge node[auto]{$d$} (m-3-4);
\path[->,font=\scriptsize] (m-1-2) edge node[auto]{$\alpha^*$} (m-2-2);
\path[->,font=\scriptsize] (m-1-3) edge node[auto]{$\alpha^*$} (m-2-3);
\path[->,font=\scriptsize] (m-1-4) edge node[auto]{$\alpha^*$} (m-2-4);
\path[->,font=\scriptsize] (m-2-2) edge node[auto]{$0^*-1^*$} (m-3-2);
\path[->,font=\scriptsize] (m-2-3) edge node[auto]{$0^*-1^*$} (m-3-3);
\path[->,font=\scriptsize] (m-2-4) edge node[auto]{$0^*-1^*$} (m-3-4);
\end{tikzpicture}
\]
Since $\beta$ represents a tame torsor $\T$ on $X$, $\alpha^*(\beta)$ represents the torsor $\alpha^*(\T)$, which is tame by Lemma~\ref{tametotame}. By Corollary~\ref{simplexex}, there exists $\gamma \in \check C^0(\Delta^1,A)$ with $d \gamma= \alpha^*(\beta)$. Since
\[
d(0^*-1^*)(\gamma)= (0^*-1^*)\alpha^*(\beta)=0,
\]
we conclude that $(0^*-1^*)(\gamma)$ lies in
\[
A = H^0(\Delta^0,A)= \ker (\check C^0(\Delta^0,A)\stackrel{d}{\to} \check C^1(\Delta^0,A)).
\]
It is easy to verify that the assignment
\[
\langle \cdot , \cdot \rangle: (a,b) \longmapsto  (0^*-1^*)(\gamma) \in A
\]
does not depend on the choices made.
By the explicit geometric relation between \v{C}ech 1-cocycles and torsors, and since our construction of finite push-forwards of torsors is compatible with the construction of transfers for $\qfh$-sheaves given in \cite{SV}, \S5, we see that the pairing constructed above coincides with the one constructed in Proposition~\ref{pairing}.

Finally, let
\begin{equation} \label{neq3}
A \ \hookrightarrow \ I^0 \to I^1 \to I^2
\end{equation}
be a (partial) injective resolution of the constant sheaf $A$ in the category of $\Z/m\Z$-module sheaves on $(\Sch/k)_{\qfh}$. Let $\phi: (\Sch/k)_{\qfh} \to  (\Sch/k)_{\et}$ denote the natural map of sites. Since $\phi^*$ is exact, $\phi_*$ sends injective sheaves to injective sheaves. By \cite{SV}, Thm.\ 10.2, we have $R^0\phi_*(A)=A$ and $R^i\phi_*(A)=0$ for $i\geq 1$. Hence \eqref{neq3} is also a partial resolution of $A$ by injective, \'{e}tale sheaves of $\Z/m\Z$-modules. We choose a quasi-isomorphism
\[
[0\to \underline{\check C}^0(- , A) \to \underline{\check C}^1(- , A) \to \underline{\check C}^2(- , A)] \lang  [0\to I^0 \to I^1 \to I^2]
\]
of truncated complexes of $\qfh$-sheaves. Since \v Cech- and \'{e}tale cohomology agree in dimension $\leq 1$, the induced map on global sections is a quasi-isomorphism of truncated complexes of abelian groups. Hence the pairing of Proposition~\ref{pairing} can also be obtained by the same procedure as above but using the diagram
\[
\begin{tikzpicture}
\matrix (m) [matrix of math nodes, row sep=2em,
column sep=2em, text height=2.5ex, text depth=1.25ex]
{& I^0(X) &   I^1(X) &  I^2(X) \\
&  I^0(\Delta^1) &   I^1(\Delta^1) & I^2(\Delta^1) \\
A& I^0(\Delta^0) &  I^1(\Delta^0) & I^2(\Delta^0). \\
};
\path[->,font=\scriptsize] (m-1-2) edge node[auto]{$d$} (m-1-3);
\path[->,font=\scriptsize] (m-1-3) edge node[auto]{$d$} (m-1-4);
\path[->,font=\scriptsize] (m-2-2) edge node[auto]{$d$} (m-2-3);
\path[->,font=\scriptsize] (m-2-3) edge node[auto]{$d$} (m-2-4);
\path[right hook->,font=\scriptsize] (m-3-1) edge node[auto]{} (m-3-2);
\path[->,font=\scriptsize] (m-3-2) edge node[auto]{$d$} (m-3-3);
\path[->,font=\scriptsize] (m-3-3) edge node[auto]{$d$} (m-3-4);
\path[->,font=\scriptsize] (m-1-2) edge node[auto]{$\alpha^*$} (m-2-2);
\path[->,font=\scriptsize] (m-1-3) edge node[auto]{$\alpha^*$} (m-2-3);
\path[->,font=\scriptsize] (m-1-4) edge node[auto]{$\alpha^*$} (m-2-4);
\path[->,font=\scriptsize] (m-2-2) edge node[auto]{$0^*-1^*$} (m-3-2);
\path[->,font=\scriptsize] (m-2-3) edge node[auto]{$0^*-1^*$} (m-3-3);
\path[->,font=\scriptsize] (m-2-4) edge node[auto]{$0^*-1^*$} (m-3-4);
\end{tikzpicture}
\]
By \cite{SV}, Theorem~10.7, the same argument applies with a partial injective resolution of the constant sheaf $A$ in the category of $\Z/m\Z$-module sheaves on $(\Sch/k)_{h}$.

\section{The case of smooth curves} \label{curvesect}
In this section we prove Theorem~\ref{main} in the case that $X=C$ is a smooth curve.

\medskip
Let $k$ be an algebraically closed field of characteristic $p\geq 0$, and let $C$ be a smooth, but not necessarily projective, curve over $k$. Let the semi-abelian variety $\cA$ be the generalized Jacobian of $C$ with respect to the modulus given by the sum of the points on the boundary of the regular compactification $\bar C$ of $C$ (cf.\ \cite{Se}, Ch.\,5).
The group $\cA(k)$ is the subgroup of degree zero elements of the relative Picard group $\Pic(\bar C,\bar C \smallsetminus C)$. By \cite{SV}, Thm.\ 3.1 (see \cite{Li}, for the case $C=\bar C$), there is an isomorphism
\[
H_0^S(C,\Z)^0:=\ker(H_0^S(C,\Z) \stackrel{\deg}{\lang} \Z)  \cong \cA(k),
\]
in particular, $\cA(k)$ is a quotient of the group of zero cycles of degree zero on $C$.
From the coefficient sequence together with the divisibility of $H_1^S(C,\Z)$ (which is isomorphic to $k^\times$ if $C$ is proper and zero otherwise), we obtain an isomorphism
\begin{equation} \label{neq4}
H_1^S(C,\Z/m\Z) \xrightarrow[\sim]{\;\delta \;} \null_m H_0^S(C,\Z) \cong \null_m \cA(k).
\end{equation}
After fixing a closed point $P_0$ of $C$, the  morphism $C \to \cA$, $P\mapsto P-P_0$,  is universal for morphisms of $C$ to semi-abelian varieties, i.e., $\cA$ is the generalized Albanese variety of $C$ (\cite{Se}, V, Th.\,2).

Consider the $m$-multiplication map $\cA\stackrel{m}{\to} \cA$. Its maximal \'{e}tale subcovering $\widetilde \cA\to \cA$ is the quotient of $\cA$ by the connected component of the finite group scheme $\null_m \cA$ (if $(p,m)=1$, the connected component is trivial). The projection $\cA\to \widetilde \cA$ induces an isomorphism $\cA(k)\stackrel{\sim}{\to} \widetilde \cA(k)$ on rational points, and we identify $\cA(k)$ and $\widetilde \cA(k)$ via this isomorphism. With respect to this identification, the projection $\widetilde A(k)\to \cA(k)$ is the $m$-multiplication map on $\cA(k)$.

By \cite{Se}, Ch.\ IV,  $\widetilde C := C\times_\cA \widetilde \cA$ is the maximal abelian tame \'{e}tale covering of $C$ with Galois group annihilated by~$m$. Because $\Aut_{\cA}(\widetilde \cA)\cong \null_m \cA(k)$, we obtain an isomorphism
\begin{equation} \label{neq5}
\Hom(\null_m\cA(k),A) \xrightarrow[\sim]{\;\tau\;} H^1_t(C,A)
\end{equation}
for any finite abelian group $A$ with $mA=0$.

\begin{theorem}\label{vergleichepaarung}  For any finite abelian group $A$ with $mA=0$, the diagram
\[
\begin{tikzpicture}[bij/.style={above,sloped,inner sep=0.5pt},ibij/.style={below,sloped,inner sep=2.5pt}]
\matrix (m) [matrix of math nodes, row sep=2em,
column sep=1em, text height=2.5ex, text depth=1.25ex]
{
H_1^S(C,\Z/m\Z) & \times & H^1_t(C,A) && &A\\
\null_m \cA(k)&\times & \Hom(\null_m\cA(k),A)&\phantom{dddd}& & A\\
};
\path[->,font=\scriptsize] (m-1-3) edge node[auto]{$\langle \ , \ \rangle $} (m-1-6);
\path[->] (m-1-1) edge node[bij]{$\sim$} node[left]{$\scriptstyle \delta$} (m-2-1);
\path[->] (m-2-3) edge node[bij]{$\sim$} node[right]{$\scriptstyle \tau$} (m-1-3);
\path[->,font=\scriptsize] (m-2-3) edge node[auto]{eval}(m-2-6);
\draw[double,double distance=1mm] (m-1-6)--(m-2-6);
\end{tikzpicture}
\]
where  $\langle \ , \ \rangle $ is the pairing from Proposition~\ref{pairing} and $\mathit{eval}$ is the evaluation map, commutes. In particular, the upper pairing is perfect and the induced homomorphism $H_1^S(C,\Z/m\Z) \to \pi_1^{t,\ab}(C)/m$ is an isomorphism.
\end{theorem}

\begin{proof}
We have to show that $\phi(\delta(\zeta)) = \langle \zeta, \tau(\phi)\rangle$
for any $\zeta \in H_1^S(C,\Z/m\Z)$ and any $\phi \in \Hom(_m \cA(k),A)$. By functoriality, it suffices to consider the universal case $A=\null_m \cA(k)$, $\phi=\id$. In this case $\tau(\id)$ is the torsor $\widetilde{\pi}: \widetilde C \to C$.

Let  $C'$ be the regular compactification of $C$.  By \cite{SV}, Thm.\ 3.1, $\delta(\zeta)\in  \null_m H_0^S(C,\Z)= \null_m \cA(k)$ is the class $[z]$ of some $z\in Z_0(C)$ (the group of zero-cycles on $C$) such that
\[
mz=\gamma^{*}(0)-\gamma^{*}(1)
\]
for some finite morphism $\gamma: C' \to \P^1$ with $C' \smallsetminus C \subset \gamma^{-1}(\infty) $.
The diagram
\[
\begin{tikzcd}
C' \smallsetminus \gamma^{-1}(\infty)\arrow[hook]{r}\dar{\gamma|_{C' {\smallsetminus} \gamma^{-1}(\infty)}} & C \arrow[hook]{r} &C'\dar{\gamma}\\
\Delta^1 \arrow[hook]{rr}&&\P^1
\end{tikzcd}
\] shows that $\gamma$ induces a finite correspondence, say $g$, from  $\Delta^1$ to $C$.  The class of $g$ in  $H_1^S(C,\Z/m\Z)$ is a pre-image of $\delta(\zeta)$ under $H_1^S(C,\Z/m\Z) \stackrel{\sim}{\to} \null_m H_0^S(C,\Z)$, i.e., $\zeta$ is represented by $g$.
It therefore suffices to show that
\[
[z]=\langle g, \widetilde C\rangle.
\]
Let $d$ be the degree of $\gamma$ and $\gamma^*(0)=  \sum_{i=1}^d P_i$, $\gamma^*(1)= \sum_{i=1}^d Q_i$.
Each point in $\gamma^*(0)$ and $\gamma^*(1)$ occurs with multiplicity divisible by $m$, in particular $d=mr$ for some integer $r$. After reindexing, we may assume that $P_i=P_{j}$ and $Q_i=Q_j$ for $i\equiv j \bmod r$, hence
\[
z= \sum_{i=1}^r P_i-\sum_{i=1}^r Q_i.
\]
On the level of closed points, $\widetilde C = C\times_{\cA} \widetilde \cA$ can be identified with the set of $a\in \widetilde \cA(k)=\cA(k)$ such that $ma=P-P_0$ for some point $P\in C$ ($a$ projects to $P$ in $C$, i.e., $\widetilde{\pi}(a)=P$).
The $\null_m \cA(k)$-principal homogeneous space $0^*g^*\widetilde C$ can be identified with the quotient of the set
\[
\prod_{i=1}^d \widetilde{\pi}^{-1}(P_i)
\]
by the action of $\null_m \cA(k)^{d-1}$ given by
\[
(\beta_1,\ldots,\beta_{d-1}) (a_1,\ldots, a_d) = (a_1 +\beta_1, a_2 -\beta_1+\beta_2,\ldots, a_d -\beta_{d-1}).
\]
We fix points $a_1,\ldots,a_d\in \widetilde C$ over $P_1,\ldots,P_d$ subject to the condition $a_i=a_j$ for $P_i=P_j$. Then $0^*g^*\widetilde C$ is identified with the quotient of the set
\[
(a_1 + \null_m \cA(k)) \times  \cdots \times  (a_d + \null_m \cA(k))
\]
by the action of $\null_m \cA(k)^{d-1}$. Since each $a_i$ occurs with multiplicity divisible by $m$,  the trivialization $0^*g^*(\widetilde C) \stackrel{\sim}{\to} \null_m \cA(k)$ given by
\[
\overline{(a_1 + \alpha_1, \ldots, a_d +\alpha_d)} \longmapsto \alpha_1 + \cdots + \alpha_d \in \null_m \cA(k)
\]
does not depend on the choice of the $a_i$.
We do the same with $1^*g^*(\widetilde C)$ by choosing $b_i\in \widetilde C$ over $Q_i$. Then we see that the tautological identification
$\Phi_{\taut}: 0^*g^*(\widetilde C) \stackrel{\sim}{\to} 1^*g^*(\widetilde C)$ is given by
\[
\overline{(a_1 + \alpha_1, \ldots, a_d +\alpha_d)} \longmapsto
\overline{(b_1 + \alpha_1, \ldots, b_d +\alpha_d)}.
\]
Now consider the morphism
\[
\Sigma: \Sym^d(C) \longrightarrow \cA, \ (x_1,\ldots,x_d) \longmapsto [\sum (x_i - P_0)].
\]
Associated with the $\null_m \cA(k)$-torsor $\widetilde C$ over $C$, we have the  $\null_m \cA(k)$-torsor $\Sym^d_{\null_m \cA(k)}(\widetilde C)$ over $\Sym^d(C)$ (cf.\ the paragraph preceding Lemma~\ref{symprod}).
The commutative diagram
\[
\begin{tikzcd}[row sep=small]
\Sym^d_{\null_m \cA(k)}(\widetilde C)\rar{\Sigma}\dar&\widetilde \cA\dar\\
\Sym^d(C)\rar{\Sigma} & \cA
\end{tikzcd}
\]
induces a map (hence an isomorphism) of $\null_m \cA(k)$-torsors $\Sym^d_{\null_m \cA(k)}(\widetilde C) \stackrel{\sim}{\to} \widetilde \cA \times_{\cA} \Sym^d(C)$.
Consider the morphism $S_g: \Delta^1_k  \to \Sym^d(C)$ associated with the finite correspondence $g$. Since the generalized Jacobian of $\Delta^1_k\cong \A^1_k$ is $\Spec(k)$, the composite
\[
\Delta^1_k \stackrel{S_g}\longrightarrow \Sym^d(C) \stackrel{\Sigma}{\longrightarrow} \cA
\]
is constant with value $a:= [ \sum_{i=1}^d (P_i - P_0)] = [\sum_{i=1}^d (Q_i -  P_0)] \in \cA(k)$.
By Lemma~\ref{symprod}, we obtain an isomorphism
\[
g^*(\widetilde C)= S_g^*(\Sym^d(\widetilde C))= \Sigma^* S_g^* \widetilde \cA = \Delta^1_k \times \widetilde{\pi}^{-1}(a)
\]
(giving a trivialization after choosing a point in $\widetilde{\pi}^{-1}(a)$).
On the fibre over $0$ it is given by
\[
\overline{(a_1 + \alpha_1, \ldots, a_d +\alpha_d)} \longmapsto \sum_{i=1}^d (a_i + \alpha_i) \in \widetilde{\pi}^{-1}(a)\subset \widetilde \cA
\]
and similarly on the fibre over $1$. We conclude that $\Phi_{\para} \circ \Phi_{\taut}^{-1}$ is translation by
\[
\sum_{i=1}^d (a_i-b_i) = \sum_{i=1}^r m(a_i-b_i) = \sum_{i=1}^r [P_i - Q_i] = [z].
\]
This concludes the proof.
\end{proof}

\section{The blow-up sequences}\label{exact-sect}

All schemes in this section are separated schemes of finite type over the spectrum of a perfect field $k$.
A curve on a scheme $X$ is a closed one-dimensional subscheme. The normalization of a curve $C$ is denoted by $\widetilde C$.

\bigskip
Now let
\[
\begin{tikzcd}
Z'\rar\dar&X'\dar{\pi}\\
Z\rar{i}& X
\end{tikzcd}
\]
be an abstract blow-up square, i.e.,  a cartesian  diagram of schemes such that $\pi: X'\to X$ is proper, $i: Z \to X$ is a closed embedding and $\pi$ induces an isomorphism $(X'\smallsetminus Z')_\textrm{red} \stackrel{\sim}{\to} (X \smallsetminus Z)_\textrm{red}$.

\begin{proposition} \label{exact-etale}  Given an abstract blow-up square and an abelian group $A$, assume that $\pi$ is finite or $A$ is torsion. Then there is a natural exact sequence
\[
0 \to H^0_\et(X,A) \to H^0_\et(X',A) \oplus H^0_\et(Z,A) \to H^0_\et(Z',A)  \qquad
\]
\[
\qquad \stackrel{\delta}{\to} H^1_t(X,A) \to H^1_t(X',A) \oplus H^1_t(Z,A) \to H^1_t(Z',A).
\]
\end{proposition}

\begin{proof}
We call an abstract blow-up square trivial, if  $i$ is surjective (i.e., $s_\mathrm{red}$ is an isomorphism) or if $\pi_\mathrm{red}: X'_\mathrm{red}\to X_\mathrm{red}$ has a section.
Every  abstract blow-up square with $X$ a connected regular curve is trivial.

Now let an arbitrary abstract blow-up square be given.
If $A$ is torsion, the proper base change theorem implies (cf.\ \cite{ge-duke}, 3.2 and 3.6) that
we have a long exact sequence
\[
 \cdots \to H^i_\et(X,A) \to H^i_\et(X',A) \oplus H^i_\et(Z,A) \to H^i_\et(Z',A) \to  H^{i+1}_\et(X,A) \to \cdots
\]
If $\pi$ is finite, the same is true for arbitrary $A$ since $\pi_*$ is exact. If the blow-up square is trivial, this long exact sequence splits into short exact sequences $0\to H^i_\et(X,A) \to H^i_\et(X',A) \oplus H^i_\et(Z,A) \to H^i_\et(Z',A) \to 0$ for all $i$.

\medskip
Next we show the exact sequence of the proposition. We omit the coefficients $A$ and put $H^0_t(X)= H^0_\et(X)$.
We first show, that the image of the boundary map $\delta: H^0_\et(Z') \to H^1_\et(X)$ has image in $H^1_t(X)$, thus showing the existence of $H^0_t(Z') \to H^1_t(X)$ and, at the same time, the exactness of the sequence at $H^1_t(X)$.
Let $\widetilde C \to X$ be the normalization of a curve in $X$. The base change
\[
\begin{tikzcd}
Z'_{\widetilde C}\rar\dar&X'_{\widetilde C}\dar{\widetilde \pi}\\
Z_{\widetilde C}\rar{\widetilde i}& {\widetilde C}
\end{tikzcd}
\]
of our abstract blow-up square to $\widetilde C$ is a trivial abstract blow-up square. Therefore, for any $\alpha \in H^0_\et(Z')$, the pull-back of  $\alpha$ to $H^0_\et(Z'_{\widetilde C})$ lies in the image of $H^0_\et(X'_{\widetilde C}) \oplus H^0_\et(Z_{\widetilde C}) \to H^0_\et(Z'_{\widetilde C}) $ and has therefore trivial image under $\delta: H^0_\et(Z'_{\widetilde C}) \to H^1_\et (\widetilde C)$. Therefore, $\delta(\alpha) \in H^1_\et(X)$ has trivial image in $H^1_\et(\widetilde C)$ for every curve $C\subset X$, in particular, it lies in $H^1_t(X)$.

It remains to show exactness at $H^1_t(X')\oplus H^1_t(Z)$. Let $\alpha$ be in this group with trivial image in $H^1_t(Z')$. Then there exists $\beta \in H^1_\et(X)$ mapping to $\alpha$ and it remains to show that $\beta$ lies in the subgroup $H^1_t(X)$. But this is clear, because for every curve $C\subset X$ we have $H^1_t(\widetilde C)=\ker(H^1_t(X'_{\widetilde C}) \oplus H^1_t(Z_{\widetilde C}) \to H^1_t(Z'_{\widetilde C}))$.
\end{proof}

\begin{proposition} \label{exact-suslin} Given an abstract blow-up square
\[
\begin{tikzcd}
Z'\rar\dar&X'\dar{\pi}\\
Z\rar{i}& X
\end{tikzcd}
\]
and an abelian group $A$, there is a natural exact sequence of Suslin homology groups
\[
H^S_1(Z',A) \to H^S_1(X',A) \oplus H^S_1(Z,A) \to H^S_1(X,A)  \qquad\qquad
\]
\[
\qquad  \stackrel{\delta}{\to} H^S_0(Z',A) \to H^S_0(X',A) \oplus H^S_0(Z,A)  \to H^S_0(X,A) \to 0.
\]
\end{proposition}

\begin{proof}
Consider the exact sequences
\[
C_\bullet(Z',A) \hookrightarrow C_\bullet(X',A)\oplus C_\bullet(Z,A) \to C_\bullet(X,A)  \twoheadrightarrow K_\bullet^{A}
\]
and
\[
C_\bullet(Z') \hookrightarrow C_\bullet(X')\oplus C_\bullet(Z) \to C_\bullet(X)  \twoheadrightarrow K_\bullet,
\]
where $K_\bullet^{A}$ and $K_\bullet$  are defined to make the sequences exact. Since the complexes $C_\bullet(-)$ consist of free abelian groups, in order the show the statement of the proposition, it suffices to show that $H_i(K_\bullet)=0$ for $i\leq 2$.
Let $\Sm/k$ be the full subcategory of $\Sch/k$ consisting of smooth schemes. For  $Y\in \Sch/k$ we consider the presheaf $c(Y)$ on $\Sm/k$ given by
$ c(Y) (U) = \Cor(U,Y)$.
Then, by \cite{blowup}, Thm.\,5.2, 4.7 and its proof, the sequence
\[
0\to c(Z') \to c(X')\oplus c(Z) \stackrel{(\pi_*,i_*)}\to c(X)
\]
is exact and $F:= \mathit{coker}(\pi_*,i_*)$ has the property that, for any $U\in \Sm/k$ of dimension $\leq 2$ and any $x\in F(U)$, there exists a proper birational morphism $\phi: V\to U$ with $V$ smooth such that $\phi^*(x)=0$. Let $F_\bullet$ be the complex of presheaves given by $F_n(U)=F(U\times \Delta^n)$ with the obvious differentials and let $(F_\bullet)_\Nis$ be the associated complex of sheaves on $(\Sm/k)_\Nis$. Then by \cite{SS}, Thm.\ 2.4, the Nisnevich sheaves
\[
\mathcal{H}_i((F_\bullet)_\Nis)
\]
vanish for $i\leq 2$. Evaluating at $U=\Spec(k)$ yields the result.
\end{proof}

Now assume that $k$ is algebraically closed. Let
\[
\rec_{1,X}: H_1^S(X,\Z/m\Z) \to H^1_t(X,\Z/m\Z)^*
\]
 be the reciprocity map constructed in Section~\ref{torseur} and let
\[
\rec_{0,X}: H_0^S(X,\Z/m\Z) \to H^0_\et(X,\Z/m\Z)^*
\]
be the homomorphism induced by the pairing
\[
\langle \cdot,\cdot\rangle: \ H_0^S(X,\Z/m\Z) \times H^0_\et(X,\Z/m\Z) \longrightarrow \Z/m\Z
\]
defined as follows: Given $a\in H_0^S(X,\Z/m\Z)$ and $b\in H^0_\et(X,\Z/m\Z)$, we represent $a$ by a correspondence $\alpha\in\Cor(\Delta^0,X)$ and put $\langle a,b\rangle= \alpha^*(b) \in H^0_\et(\Delta^0,\Z/m\Z)\cong\Z/m\Z$. This is well-defined since the homomorphisms $0^*, 1^*:H^0_\et(\Delta^1,\Z/m\Z) \to H^0_\et(\Delta^0,\Z/m\Z)$ agree.

\begin{lemma}\label{rec0iso}
For any $m$, $\rec_{0,X}$ is an isomorphism.
\end{lemma}

\begin{proof}
For connected $X$, we have the commutative diagram
\[
\begin{tikzcd}
H_0^S(X,\Z/m\Z)\dar{\deg}\rar{\rec_{0,X}}& H^0_\et(X,\Z/m\Z)^*\dar{\wr}\\
H_0^S(k,\Z/m\Z)\rar{\rec_{0,k}}[swap]{\sim}& H^0_\et(k,\Z/m\Z)^*.
\end{tikzcd}
\]
Hence, for connected $X$,  it suffices by functoriality
to consider the mod $m$ degree map. In particular, $\rec_{0,X}$ is surjective for arbitrary $X$ and is an isomorphism if $\dim X =0$. If $X$ is a smooth connected curve, then $H_0^S(X,\Z)=\Pic(\bar X,\bar X \smallsetminus X)$, where $\bar X$ is the smooth compactification of $X$ (cf.\ \cite{SV}, Thm.\ 3.1). The subgroup $\Pic^0(\bar X,\bar X \smallsetminus X)$ of degree zero elements is the group of $k$-rational points of the Albanese of $X$, and hence divisible. Therefore, $\rec_{0,X}$ is an isomorphism for connected, and hence for all smooth  curves.
Considering the normalization morphism of an arbitrary scheme of dimension~$1$ and the  exact sequences of Propositions~\ref{exact-etale} and \ref{exact-suslin}, the five-lemma shows that $\rec_{0,X}$ is a isomorphism for $\dim X \le 1$.

It remains to show that $\rec_{0,X}$ is injective for arbitrary $X$. We may assume $X$ to be connected. Let $a \in \ker (\rec_{0,X})$ and let $\alpha\in Z_0(X)$ be a representing $0$-cycle. Since $\mathit{supp} (\alpha)$ is finite, we can find a connected  $1$-dimensional closed subscheme $Z\subset X$ containing $\mathit{supp} (\alpha)$ (use, e.g., \cite{Mu}, II \S 6 Lemma). Since $\rec_{0,Z}$ is injective and  $a$ is in the image of $H_0^S(Z,\Z/m\Z) \to H_0^S(X,\Z/m\Z)$, we conclude that $a=0$.
\end{proof}

\begin{corollary}\label{kerdeg}
Let $k$ be an algebraically closed field and let $X\in \Sch/k$ be connected. Then the kernel of the degree map
\[
\deg: H_0^S(X,\Z) \lang H_0^S(k,\Z)\cong \Z
\]
is divisible.
\end{corollary}

\begin{proposition} \label{compare_ex_seq} Let $k$ be algebraically closed and let
\[
\begin{tikzcd}
Z'\rar{i'}\dar{\pi'}&X'\dar{\pi}\\
Z\rar{i}& X
\end{tikzcd}
\]
be an abstract blow-up square. Then for any integer $m\geq 1$ the diagram
\[
\begin{tikzpicture}
\matrix (m) [matrix of math nodes, row sep=2em,
column sep=2em, text height=2.5ex, text depth=1.25ex]
{H_1^S(X,\Z/m\Z) & H_0^S(Z',\Z/m\Z) \\
H^1_t(X,\Z/m\Z)^* &  H^0_\et(Z',\Z/m\Z)^*,  \\
};
\path[->,font=\scriptsize] (m-1-1) edge node[auto]{$\delta$} (m-1-2);
\path[->,font=\scriptsize] (m-1-1) edge node[auto]{$\rec_{1,X}$} (m-2-1);
\path[->,font=\scriptsize] (m-2-1) edge node[auto]{$\delta^*$} (m-2-2);
\path[->,font=\scriptsize] (m-1-2) edge node[auto]{$\rec_{0,X}$} (m-2-2);
\end{tikzpicture}
\]
commutes. Here $\delta$ is the boundary map of Proposition~\ref{exact-suslin} and $\delta^*$ is the dual of the boundary map of Proposition~\ref{exact-etale}.
\end{proposition}

\begin{proof} We have to show that the diagram
\[
\begin{tikzpicture}[bij/.style={above,sloped,inner sep=0.5pt},ibij/.style={below,sloped,inner sep=2.5pt}]
\matrix (m) [matrix of math nodes, row sep=2em,
column sep=1em, text height=2.5ex, text depth=1.25ex]
{
H_1^S(X,\Z/m\Z) & \times & H^1_t(X,\Z/m\Z) && &\Z/m\Z\\
H_0^S(Z',\Z/m\Z) & \times & H^0_\et(Z',\Z/m\Z) && &\Z/m\Z\\
};
\path[->,font=\scriptsize] (m-1-3) edge node[auto]{$\langle \ , \ \rangle $} (m-1-6);
\path[->,font=\scriptsize] (m-2-3) edge node[auto]{$\langle \ , \ \rangle $} (m-2-6);
\path[->,font=\scriptsize] (m-1-1) edge  node[left]{$\delta$} (m-2-1);
\path[->,font=\scriptsize] (m-2-3) edge  node[right]{$\delta$} (m-1-3);
\draw[double,double distance=1mm] (m-1-6)--(m-2-6);
\end{tikzpicture}
\]
commutes.  Given $a\in H_1^S(X,\Z/m\Z)$ and $ b\in H^0_\et(Z',\Z/m\Z)$, we choose a representing correspondence $\alpha \in C_1(X,\Z/m\Z)=\Cor(\Delta^1, X)\otimes \Z/m\Z$ in such a way that it has a pre-image $\widehat \alpha \in C_1(X',\Z/m\Z) \oplus C_1( Z,\Z/m\Z)$ (see the proof of Proposition~\ref{exact-suslin}). By definition, $\delta a \in H^S_0(Z',\Z/m\Z)$ is represented by a correspondence $\gamma \in C_0(Z',\Z/m\Z)$ such that the diagram
\[
\begin{tikzpicture}
\matrix (m) [matrix of math nodes, row sep=2em,
column sep=2em, text height=2.5ex, text depth=1.25ex]
{\Delta^0 & \Delta^1 \\
Z' &  X' \amalg Z \\
};
\path[->,font=\scriptsize] (m-1-1) edge node[auto]{$0-1$} (m-1-2);
\path[->,font=\scriptsize] (m-1-1) edge node[left]{$\gamma$} (m-2-1);
\path[->,font=\scriptsize] (m-2-1) edge node[auto]{$i'-\pi'$} (m-2-2);
\path[->,font=\scriptsize] (m-1-2) edge node[auto]{$\widehat \alpha$} (m-2-2);
\end{tikzpicture}
\]
of correspondences commutes modulo~$m$. Next choose an  injective resolution $\Z/m\Z \to I^\bullet$  of $\Z/m\Z$ in the category of sheaves of $\Z/m\Z$-modules on $(\Sch/k)_h$  in order to compute the pairings (cf.\ the end of section \ref{chech-sect}). Consider the following diagram

\[
\begin{tikzcd}
\phantom{x}&&I^0(X')\oplus I^0(Z)\arrow[bend right]{lldd}[swap]{\widehat\alpha^*} \dar{d}\rar{i'^*-\pi'^*}&I^0(Z')\dar{d}\\
&I^1(X) \arrow[hookrightarrow]{r}{(\pi^*,i^*)}\dar{\alpha^*}& I^1(X')\oplus I^1(Z)\rar{i'^*-\pi'^*} \dar{\widehat\alpha^*} &I^1(Z')\\
I^0(\Delta^1)\rar{d}\dar{0^*-1^*}&I^1(\Delta^1) \dar{0^*-1^*}\arrow[dash,yshift=0.3ex]{r}\arrow[dash,yshift=-0.3ex]{r} &I^1(\Delta^1)\dar{0^*-1^*}\rar{d}&I^2(\Delta^1)\dar{0^*-1^*}\\
I^0(\Delta^0)\rar{d}&I^1(\Delta^0)\arrow[dash,yshift=0.3ex]{r}\arrow[dash,yshift=-0.3ex]{r}&I^1(\Delta^0)\rar{d}&I^2(\Delta^0)\\
\end{tikzcd}
\]
By the argument of \cite{MVW} Lemma 12.7, the sequence \[
0\to F(X) \to F(X')\oplus F(Z) \to F(Z')
\]
is exact for every $h$-sheaf $F$. Therefore the second line in the diagram is exact.
The proper base change theorem implies (cf.\ \cite{ge-duke}, 3.2 and 3.6) that
\[
 I^\bullet(X) \longrightarrow I^\bullet(X')\oplus I^\bullet(Z) \longrightarrow I^\bullet(Z') \stackrel{[1]}{\longrightarrow}
\]
is an exact triangle in $D(Ab)$. For the exact sequence of complexes
\[
0\to I^\bullet(X) \to I^\bullet(X')\oplus I^\bullet(Z) \to I^\bullet(Z') \to \mathit{coker}^\bullet \to 0,
\]
this implies that the complex $\mathit{coker}^\bullet$ is exact.  Therefore, $b \in \ker(I^0(Z')\to I^1(Z'))$ has a pre-image $\widehat \beta \in I^0(X')\oplus I^0(Z)$.
Then
\[
d\widehat \beta \in \ker (I^1(X')\oplus I^1(Z) \to I^1(Z')),
\]
and  there exists a unique $\varepsilon\in I^1(X)$ with $(\pi^*,i^*)(\varepsilon)=d\widehat \beta$  representing $\delta b \in H^1_t(X)$. We see that $\widehat \alpha^* (d\widehat \beta)= \alpha^* (\varepsilon)$.
It follows that
\[
d(\widehat \alpha^*(\widehat \beta))=\widehat \alpha^*(d\widehat \beta)= \alpha^*(\varepsilon) \in \ker(I^1(\Delta^1)\xrightarrow{0^*-1^*} I^1(\Delta^0)).
\]
By definition of $\langle \ , \ \rangle$, we obtain
\[
\langle a, \delta(b)\rangle = (0^*-1^*)\widehat \alpha^*\widehat \beta \in \ker(I^0(\Delta^0)\to I^1(\Delta^0))=\Z/m\Z.
\]
On the other hand, $\langle \delta a, b\rangle= \gamma^*(b) \in H^0_\et(\Delta^0)$ is represented by $\gamma^*\beta\in I^0(\Delta^0)$ and the commutative diagram of correspondences above implies
\[
\gamma^*\beta = \gamma^*(i'^*-\pi'^*)(\widehat \beta) = (0^*-1^*)\widehat \alpha^*\widehat \beta.
\]
This finishes the proof.
\end{proof}

\begin{proposition}\label{goodcurve}
Let $X$ be a normal, generically smooth, connected scheme of finite type over a field $k$ and let $M\subseteq H^1_\et(X,\Z/m\Z)$ be a finite subgroup. Then there exists a regular curve $C$ over $k$ and a finite morphism $\phi: C \to X$ such that $M$ has trivial intersection with the kernel of $\phi^*: H^1_\et(X,\Z/m\Z) \to H^1_\et(C,\Z/m\Z)$.
\end{proposition}

\begin{proof}
For any normal scheme $Z$ and  dense open subscheme $Z'\subset Z$, the induced map $H^1_\et(Z,\Z/m\Z) \to H^1_\et(Z',\Z/m\Z)$ is injective. Hence we may replace $X$ by an open subscheme and assume that $X$ is smooth. Let $Y \to X$ be the finite abelian \'{e}tale covering corresponding to the kernel of $\pi_1^\ab(X) \twoheadrightarrow M^*$. We have to find a regular curve $C$ and a finite morphism $C\to X$ such that $C\times_XY$ is connected.

Choose a separating transcendence basis $t_1,\ldots, t_d$ of $k(X)$ over $k$. This yields a rational map $X \to \P^d_{k}$.
Let $t$ be another indeterminate and let $X_t$ (resp.\ $Y_t$) be the base change of $X$ (resp.\ $Y$) to the rational function field $k(t)$.  Consider the composition $\phi: Y_t\to X_t \to \P^d_{k(t)}$. Since $k(t)$ is Hilbertian \cite{FJ}, Thm.\ 12.10, we can find a rational point $P\in \P^d_{k(t)}$ over which $\phi$ is defined and such that $P$ has exactly one pre-image $y_t$ in $Y_t$. The image $x_t \in X_t$ of $y_t$ has exactly one pre-image in $Y_t$. Let $x$ be the image of $x_t$ in $X$. If $\mathrm{trdeg}_k k(x)=1$ put $x'=x$, if $\mathrm{trdeg}_k k(x)=0$ (i.e., $x$ is a closed point in $X$) choose any $x'\in X$ with $\mathrm{trdeg}_k k(x')=1$ such that $x$ is a regular point of the closure of $x'$. In both cases the normalization $C$ of the closure of $x'$ in $X$ is a regular curve with the desired property.
\end{proof}

\section{Proof of the main theorem}\label{proof-sect}

In this section we prove our main result. We say that ``resolution of singularities holds for schemes of dimension $\leq d$ over $k$'' if the following two conditions are satisfied.

\medskip
\begin{compactitem}
\item[\rm (1)] For any integral separated scheme of finite type $X$ of dimension $\leq d$ over~$k$,  there exists a
projective birational morphism $Y \to X$ with $Y$ smooth over~$k$ which is an isomorphism over the regular locus of $X$.
\item[\rm (2)] For any integral smooth scheme $X$ of dimension $\leq d$ over $k$ and any birational proper morphism
$Y \to X$ there exists a tower of morphisms $X_n \to X_{n-1} \to \dots \to X_0= X$, such that $X_n \to X_{n-1}$  is a blow-up with a smooth center for $i=1,\ldots,n$, and such that the composite morphism $X_n \to X$ factors through $Y \to X$.
\end{compactitem}

\begin{theorem}[\bf =Theorem~\ref{main}]\label{main-besser} Let $k$ be an algebraically closed field of characteristic $p\ge 0$, $X$ a separated scheme of finite type over $k$ and $m$ a natural number. Then
\[
\rec_X: H_1^S(X,\Z/m\Z) \lang \pi_1^{t,\ab}(X)/m
\]
is surjective. It is an isomorphism of finite abelian groups if\/ $(m,p)=1$, and for general $m$ if resolution of singularities holds for schemes of dimension $\leq \dim X + 1$ over~$k$.
\end{theorem}

The proof will occupy the rest of this section. Following the notation of Section~\ref{exact-sect}, we write $H^0_t=H^0_\et$ and consider the maps
\[
\rec_{i,X}: H_i^S(X,\Z/m\Z) \to H^i_t(X,\Z/m\Z)^*
\]
for $i=0,1$ (i.e., $\rec_X=\rec_{1,X}$).
Given a morphism $X'\to X$, we have a commutative diagram of pairings defining $\rec_i$ for $i=0,1$.

\[
\begin{tikzpicture}[bij/.style={above,sloped,inner sep=0.5pt},ibij/.style={below,sloped,inner sep=2.5pt}]
\matrix (m) [matrix of math nodes, row sep=2em,
column sep=1em, text height=2.5ex, text depth=1.25ex]
{
H_i^S(X',\Z/m\Z) & \times & H^i_t(X',\Z/m\Z) && &\Z/m\Z\,\\
H_i^S(X,\Z/m\Z) & \times & H^i_t(X,\Z/m\Z) && &\Z/m\Z.\\
};
\path[->,font=\scriptsize] (m-1-3) edge node[auto]{$\langle \ , \ \rangle $} (m-1-6);
\path[->,font=\scriptsize] (m-2-3) edge node[auto]{$\langle \ , \ \rangle $} (m-2-6);
\path[->] (m-1-1) edge  node[left]{$\scriptstyle \phi_*$} (m-2-1);
\path[->] (m-2-3) edge  node[right]{$\scriptstyle \phi^*$} (m-1-3);
\draw[double,double distance=1mm] (m-1-6)--(m-2-6);
\end{tikzpicture}
\]

\medskip\noindent
\emph{Step 1}: $\rec_{1,X}$ is surjective for arbitrary $X$.

\medskip\noindent
We may assume that $X$ is reduced and proceed by induction on $d=\dim X$. The case $\dim X=0$ is trivial. Consider the  normalization morphism $X' \to X$, which is an isomorphism outside a closed subscheme $Z\subset X$ of dimension $\leq d-1$.
Using the exact sequences of Propositions~\ref{exact-etale} and \ref{exact-suslin}, which are compatible by Proposition~\ref{compare_ex_seq} and the fact that $\rec_{0,X}$ is an isomorphism by Lemma~\ref{rec0iso}, a diagram chase shows that it suffices to show surjectivity of $\rec_{1,X}$ for normal schemes.

Let $X$ be normal.  Since $H^1_t(X,\Z/m\Z)$ is finite, it suffices to show that the pairing defining $\rec_{1,X}$ has a trivial right kernel. We may assume that $X$ is connected. Let $b\in H^1_t(X,\Z/m\Z)$ be arbitrary but non-zero. By Proposition~\ref{goodcurve}, we find a morphism $\phi: C \to X$ with $C$ a smooth curve such that $\phi^*(b)\in H^1_\et(C,\Z/m\Z)$ is non-zero. Since the pairing for $C$ is perfect by Theorem~\ref{vergleichepaarung}, the pairing for $X$ has a trivial right kernel.

\medskip\noindent
\emph{Step 2}: Theorem~\ref{main-besser} holds if  $(m,p)=1$.

\medskip\noindent
If $(m,p)=1$,  $H_1^S(X,\Z/m\Z)$ and $H^1_\et(X,\Z/m\Z)^*$ are isomorphic finite abelian groups by \cite{SV}. In particular, they have the same order. Hence the surjective homomorphism $\rec_{1,X}$ is an isomorphism.

\medskip\noindent
\emph{Step 3}: Theorem~\ref{main-besser} holds for arbitrary $X$ if $m=p^r$ and resolution of singularities holds for schemes of dimension $\leq \dim X +1$ over $k$.

\medskip\noindent
We may assume that $X$ is reduced. Using resolution of singularities and Chow's Lemma, we obtain a morphism $X'\to X$ with $X'$ smooth and quasi-projective, which is an isomorphism over a dense open subscheme of $X$.  Using the exact sequences of Propositions~\ref{exact-etale} and \ref{exact-suslin}, Lemma~\ref{rec0iso}, \emph{Step~1}, induction on the dimension and the  five-lemma,  it suffices to show the result for smooth, quasi-projective schemes.

Let $X$ be smooth, quasi-projective and let $\bar X$ be a smooth, projective variety containing $X$ as a dense open subscheme. Then, by \cite[\S 5]{ge-annals}, we have an isomorphism
\[
H^1_\et(\bar X,\Z/p^r\Z)^* \cong  \mathrm{CH}_0(\bar X,1,\Z/p^r\Z).
\]
Furthermore, by \cite[\textrm{Thm.\ 2.7}]{SS} (which makes the assumption $\dim \leq 2$ but does not use it in its proof), we have an isomorphism
\[
\mathrm{CH}_0(\bar X,1,\Z/p^r\Z) \cong H_1^S(\bar X,\Z/p^r\Z).
\]
By Proposition~\ref{suslin-mod p} below, the natural homomorphism
\[
H_1^S(X ,\Z/p^r\Z) \to H_1^S(\bar X ,\Z/p^r\Z)
\]
is an isomorphism of finite abelian groups and by Proposition~\ref{numerical}, we have an isomorphism
\[
H^1_\et(\bar X,\Z/p^r\Z) \stackrel{\sim}{\to} H^1_t(X,\Z/p^r\Z).
\]
Hence the finite abelian groups $H^1_t(X,\Z/p^r\Z)^*$ and $H_1^S(X ,\Z/p^r\Z)$ are isomorphic, in particular, they have the same order. Since $\rec_{1,X}$ is surjective, it is an isomorphism.

\bigskip
In order to conclude the proof of Theorem~\ref{main-besser} it remains to show

\begin{proposition}\label{suslin-mod p}
Let $k$ be a perfect field, $X\in \Sch/k$ smooth,  $U\subset X$ a dense open subscheme and $n\geq 0$ an integer. Assume that resolution of singularities holds for schemes of dimension $\leq \dim X +n$ over $k$. Then for any $r\geq 1$ the natural map
\[
H_i^S(U ,\Z/p^r\Z) \to H_i^S(X ,\Z/p^r\Z)
\]
is an isomorphism of finite abelian groups for $i=0,\ldots,n$.
\end{proposition}
\begin{remark}
A proof of Proposition~\ref{suslin-mod p} for $n=1$ and $k$ algebraically closed  independent of the assumption on resolution of singularities would relax the condition in Theorem~\ref{main-besser} to:

There exists a smooth, projective scheme $\bar X'\in \Sch/k$, dense open subschemes $U'\subset X'\subset \bar X'$, $U\subset X$, and a surjective, proper morphism $X'\to X$ which induces an isomorphism $U'_\mathrm{red} \to U_\mathrm{red}$ .

\noindent
In particular, Theorem~\ref{main-besser} would hold for $\dim X\leq 3$ without any assumption on resolution of singularities \cite{CV}.
\end{remark}

\begin{proof}[Proof of Proposition~\ref{suslin-mod p}]
We set $R=\Z/p^r\Z$. By \cite{MVW}, Lecture 14, we have
\[
H_i^S(X,R) = \Hom_\DM(R [i], M(X,R)).
\]
Let $d=\dim X$.
Choose a series of open subschemes $U= X_d \subset \cdots  X_1 \subset  X_0=X$ such that
$Z_j:= X_{j}\smallsetminus X_{j+1}$ is smooth of dimension $j$ for $j=0,\ldots, d-1$. Using the exact Gysin triangles \cite[15.15]{MVW}
\[
M(X_{j+1},R) \to M(X_j,R) \to M(Z_j,R)(d-j)[2d-2j] \stackrel{[1]}{\to} M(X_{j+1},R)[1]
\]
 and induction, it suffices to show that
\[
\Hom_\DM (R[i],M(Z_j, R)(s)[2s]) = 0
\]
for $j=0,\ldots ,d-1$, $i=0,\ldots, n+1$ and $s\ge 1$. Using smooth compactifications of the $Z_j$ and induction again, it suffices to show
\[
\Hom_\DM (R[i],M(Z, R)(s)[2s]) = 0
\]
for $Z$ connected, smooth, projective, $i=0,\ldots, d-d_Z +n$ and $s\geq 1$.

By the comparison of higher Chow groups and motivic cohomology \cite{allagree} and by \cite{GL}, Thm.~8.5, the restriction of $R(s)$ to the small Nisnevich site of a smooth scheme $Y$ is isomorphic to $\nu_r^s[-s]$, where $\nu_r^s$ is the logarithmic de Rham Witt sheaf of Milne and Illusie. In particular, $R(s)|_Y$ is trivial for $s>\dim Y$.

For an \'{e}tale $k$-scheme $Z$ we obtain
\[
\Hom_\DM (R[i],M(Z, R)(s)[2s]) = H^{2s-i}_\Nis(Z, R(s))=0
\]
for $s\geq 1$ and all $i\geq 0$. Now assume $\dim Z \geq 1$. Using resolution of singularities for schemes of dimension $\le d+n$, the same method as in the proof of  \cite{SS}, Thm.\ 2.7 yields isomorphisms
\[
\Hom_\DM(R[i],M(Z,R))\cong \mathrm{CH}^{d_Z}(Z,i,R)
\]
for  $i=0,\ldots,d-1+n$. Applying this to $Z\times \P^s$ and using the decompositions given by the projective bundle theorem on both sides implies isomorphisms
\[
\Hom_\DM(R[i],M(Z,R)(s)[2s])\cong \mathrm{CH}^{d_Z+s}(Z,i,R)
\]
for $i=0,\ldots, d-1+n$.
By \cite{allagree}, the latter group is isomorphic to
\[
\Hom_\DM(M(Z,R)[2d_Z+2s-i], R(d_Z+s))\cong H^{2d_Z+2s-i}_\Nis(Z, R(d_Z+s)),
\]
which vanishes for $s\geq 1$. This finishes the proof.
\end{proof}

\begin{remark}
The assertion of Proposition~\ref{suslin-mod p} remains true for non-smooth $X$ if $U$ contains the singular locus of $X$ (see \cite{ge-docu}, Prop.\,3.3).
\end{remark}

\section{Comparison with the isomorphism of Suslin-Voe\-vod\-sky}\label{appendix}

\begin{theorem} Let $k$ be an algebraically closed field, $X\in\Sch/k$ and $m$ an integer prime to $\ch(k)$. Then the reciprocity isomorphism
\[
\rec_X: H_1^S(X,\Z/m\Z) \lang \pi_1^\ab(X)/m
\]
is the dual of the isomorphism
\[
\alpha_X: H^1_\et(X,\Z/m\Z) \lang H^1_S(X,\Z/m\Z)
\]
of\/ \cite{SV}, Cor.~7.8.
\end{theorem}

The proof will occupy the rest of this section.
Let $i: \Z/m\Z \hookrightarrow I^0$  be an injection into an injective sheaf in the category of $\Z/m\Z$-module sheaves on $(\Sch/k)_\qfh$ and put $J^1=\mathrm{coker} (i)$. Then (see the end of section \ref{chech-sect}) the pairing between $H_1^S(X,\Z/m\Z)$ and $H^1_\et(X,\Z/m\Z)$ constructed in Proposition~\ref{pairing} can be given as follows: For $a\in H_1^S(X,\Z/m\Z)$ choose a representing correspondence $\alpha\in \Cor(\Delta^1,X)$ and for $b\in H^1_\et(X,\Z/m\Z)$ a pre-image $\beta \in J^1(X)$. Consider the diagram
\[
\begin{tikzpicture}
\matrix (m) [matrix of math nodes, row sep=2em,
column sep=2em, text height=2.5ex, text depth=1.25ex]
{&  I^0(X) &   J^1(X) \\
&  I^0(\Delta^1) &   J^1(\Delta^1) &&(8) \\
\Z/m\Z&   I^0(\Delta^0) &   J^0(\Delta^0). \\
};
\path[->,font=\scriptsize] (m-1-2) edge node[auto]{} (m-1-3);
\path[->>,font=\scriptsize] (m-2-2) edge node[auto]{} (m-2-3);
\path[right hook->,font=\scriptsize] (m-3-1) edge node[auto]{} (m-3-2);
\path[->,font=\scriptsize] (m-3-2) edge node[auto]{} (m-3-3);
\path[->,font=\scriptsize] (m-1-2) edge node[auto]{$\alpha^*$} (m-2-2);
\path[->,font=\scriptsize] (m-1-3) edge node[auto]{$\alpha^*$} (m-2-3);
\path[->,font=\scriptsize] (m-2-2) edge node[auto]{$0^*-1^*$} (m-3-2);
\path[->,font=\scriptsize] (m-2-3) edge node[auto]{$0^*-1^*$} (m-3-3);
\end{tikzpicture}
\]
Then $\alpha^*(\beta)$ is the image of some element $\gamma \in I^0(\Delta^1)$ and $(0^*-1^*)(\gamma) \in \Z/m\Z=\ker (I^0(\Delta^0) \to J^1(\Delta^0))$ equals $\langle a, b\rangle$.

\bigskip
For $Y\in \Sch/k$ let  $\Z^\qfh_Y$  be the free $\qfh$-sheaf generated by $Y$. We set $A=\Z[1/\ch(k)]$ and $L_Y= \Z^\qfh_Y \otimes A$.  For smooth $U$ the homomorphism
\[
\Cor(U,X) \otimes A  \to \Hom_\qfh(L_U, L_X)
\]
is an isomorphism by \cite{SV}, Thm.\,6.7.
We have
\[
H^1_\et(X,\Z/m\Z)=H^1_{\qfh}(X,\Z/m\Z)= \Ext_{\qfh}^1(L_X,\Z/m\Z)
\]
\[=\mathrm{coker}( \Hom_\qfh(L_X,I^0) \to \Hom_\qfh(L_X,J^1)).
\]
The diagram $(8)$ can be rewritten in terms of Hom-groups as follows:
\[
\begin{tikzpicture}
\matrix (m) [matrix of math nodes, row sep=2em,
column sep=2em, text height=2.5ex, text depth=1.25ex]
{&  \Hom_\qfh(L_X,I^0) &   \Hom_\qfh(L_X,J^1)\, \\
&  \Hom_\qfh(L_{\Delta^1},I^0) &   \Hom_\qfh(L_{\Delta^1},J^1)\,  &(9)\\
\Z/m\Z&  \Hom_\qfh(L_{\Delta^0},I^0) &   \Hom_\qfh(L_{\Delta^0},J^1). \\
};
\path[->,font=\scriptsize] (m-1-2) edge node[auto]{} (m-1-3);
\path[->>,font=\scriptsize] (m-2-2) edge node[auto]{} (m-2-3);
\path[right hook->,font=\scriptsize] (m-3-1) edge node[auto]{} (m-3-2);
\path[->,font=\scriptsize] (m-3-2) edge node[auto]{} (m-3-3);
\path[->,font=\scriptsize] (m-1-2) edge node[auto]{$\alpha^*$} (m-2-2);
\path[->,font=\scriptsize] (m-1-3) edge node[auto]{$\alpha^*$} (m-2-3);
\path[->,font=\scriptsize] (m-2-2) edge node[auto]{$0^*-1^*$} (m-3-2);
\path[->,font=\scriptsize] (m-2-3) edge node[auto]{$0^*-1*$} (m-3-3);
\end{tikzpicture}
\]
We denote the morphism $L_X \to J^1$ corresponding to $\beta\in J^1(X)\cong\Hom_\qfh(L_X,J^1)$  by the same letter $\beta$. Putting  $E:= I^0 \times_{J_1,\beta} L_X$, the extension
\[
0\longrightarrow \Z/m\Z \longrightarrow E \longrightarrow L_X \longrightarrow 0
\]
represents $b\in \Ext^1_\qfh (L_X,\Z/m\Z)$. Consider the diagram
\[
\begin{tikzpicture}
\matrix (m) [matrix of math nodes, row sep=2em,
column sep=2em, text height=2.5ex, text depth=1.25ex]
{&  \Hom_\qfh(L_X,E) &   \Hom_\qfh(L_X,L_X)\, \\
&  \Hom_\qfh(L_{\Delta^1},E) &   \Hom_\qfh(L_{\Delta^1},L_X)\,  &(10)\\
\Z/m\Z&  \Hom_\qfh(L_{\Delta^0}, E) &   \Hom_\qfh(L_{\Delta^0}, L_X). \\
};
\path[->,font=\scriptsize] (m-1-2) edge node[auto]{} (m-1-3);
\path[->>,font=\scriptsize] (m-2-2) edge node[auto]{} (m-2-3);
\path[right hook->,font=\scriptsize] (m-3-1) edge node[auto]{} (m-3-2);
\path[->,font=\scriptsize] (m-3-2) edge node[auto]{} (m-3-3);
\path[->,font=\scriptsize] (m-1-2) edge node[auto]{$\alpha^*$} (m-2-2);
\path[->,font=\scriptsize] (m-1-3) edge node[auto]{$\alpha^*$} (m-2-3);
\path[->,font=\scriptsize] (m-2-2) edge node[auto]{$0^*-1^*$} (m-3-2);
\path[->,font=\scriptsize] (m-2-3) edge node[auto]{$0^*-1^*$} (m-3-3);
\end{tikzpicture}
\]
Because diagram (10) maps to diagram (9) via $\beta_*$ and  $\id \in \Hom_\qfh(L_X,L_X)$ maps under $\beta_*$ to $\beta \in \Hom_\qfh (L_X, J^1)$, we can calculate the pairing using diagram (10) after replacing $\beta$ by $\id$. Since $\id$ maps to $\alpha\in \Hom_\qfh(L_{\Delta^1},L_X)$ under $\alpha^*$, we see,
writing the lower part of diagram (10) in the form
\[
\begin{tikzpicture}
\matrix (m) [matrix of math nodes, row sep=2em,
column sep=2em, text height=2.5ex, text depth=1.25ex]
{\Z/m\Z&  E(\Delta^1) &   L_X(\Delta^1)\, \\
\Z/m\Z&  E(\Delta^0) &   L_X(\Delta^0),&\\
};
\path[->>,font=\scriptsize] (m-1-2) edge node[auto]{} (m-1-3);
\path[right hook->,font=\scriptsize] (m-2-1) edge node[auto]{} (m-2-2);
\path[right hook->,font=\scriptsize] (m-1-1) edge node[auto]{} (m-1-2);
\path[->,font=\scriptsize] (m-2-2) edge node[auto]{} (m-2-3);
\path[->,font=\scriptsize] (m-1-2) edge node[left]{$0^*-1^*$} (m-2-2);
\path[->,font=\scriptsize] (m-1-3) edge node[auto]{$0^*-1^*$} (m-2-3);
\path[->,font=\scriptsize] (m-1-1) edge node[left]{$0$} (m-2-1);
\path[->,dashed, font=\scriptsize] (m-1-3) edge node[auto]{$h$} (m-2-2);
\end{tikzpicture} \raisebox{1.2cm}{(11)}
\]
that
\[
\langle a, b\rangle= h(\alpha) \bmod m \in \ker (E(\Delta^0)/m \to   L_X(\Delta^0)/m)=\Z/m\Z,
\]
where $h$ is the unique homomorphism making diagram $(11)$ commutative.
We consider the complex $C_{\bullet}(X)=\Cor(\Delta^{\bullet},X) \otimes A=L_X(\Delta^\bullet)$ with the obvious differentials. By the above considerations,  the homomorphism induced by the pairing of Proposition~\ref{pairing}

\bigskip
$
H^1_\et(X,\Z/m\Z)=H^1_\qfh(X,\Z/m\Z) \longrightarrow
$
\[
\qquad H_1^S(X,\Z/m\Z)^* = \Ext^1_{A} (C_{\bullet}(X), \Z/m\Z)=\Hom_{D(A)} (C_{\bullet}(X), \Z/m\Z[1]),
\]
is given by sending an extension class
 $[\Z/m\Z \hookrightarrow E \twoheadrightarrow L_X]$ to the morphism  $C_\bullet(X)\to \Z/m\Z[1]$ in the derived category of $A$-modules represented by the morphism
\[
C_{\bullet}(X) \to [0 \to E(\Delta^0) \to L_X(\Delta^0) \to 0]
\]
which is given by $\id: L_X(\Delta^0)\to L_X(\Delta^0)$ in degree zero and by $h: L_X(\Delta^1)\to E(\Delta^0)$ in degree one.

The same construction works for any $\qfh$-sheaf of $A$-modules $F$ instead of $L_X$, i.e., setting $C_\bullet(F)=F(\Delta^\bullet)$ and starting from an element
\[
[\Z/m\Z \hookrightarrow E \twoheadrightarrow F] \in \Ext^1_\qfh(F,\Z/m\Z),
\]
we get a map $C_{\bullet}(F) \to \Z/m\Z[1]$ in the derived category of $A$-modules. We thus constructed a homomorphism
\[
\Ext^1_\qfh (F, \Z/m\Z) \longrightarrow \Ext^1(C_{\bullet}(F),\Z/m\Z), \eqno (12)
\]
which for $F=L_X$ and under the canonical identifications coincides with the map
\[
H^1_\et(X,\Z/m\Z) \longrightarrow H_1^S(X,\Z/m\Z)^*
\]
induced by the pairing constructed in Proposition~\ref{pairing}.

\bigskip
Now we compare the map $(12)$ with the map
\[
\alpha_X: \Ext^1_\qfh (F, \Z/m\Z) \longrightarrow \Ext^1_A(C_{\bullet}(F),\Z/m\Z) \eqno (13)
\]
constructed by Suslin-Voevodsky \cite{SV} (cf.\ \cite{Ge-alt} for the case of positive characteristic).
Let  $F_{\bullet}^\sim$ be the complex of $\qfh$-sheaves associated with the complex of presheaves $F_{\bullet}(U)= F(U\times \Delta^{\bullet})$. By \cite{SV}, the inclusion $F \to F_{\bullet}^\sim$ induces an isomorphism
\[
\Ext^1_\qfh(F_{\bullet}^\sim,\Z/m\Z)\stackrel{\sim}{\longrightarrow}\Ext^1_\qfh (F, \Z/m\Z),\eqno (14)
\]
and evaluation at $\Spec(k)$ induces an isomorphism
\[
\Ext^1_\qfh(F_{\bullet}^\sim,\Z/m\Z) \stackrel{\sim}{\longrightarrow} \Ext^1_A(C_{\bullet}(F),\Z/m\Z). \eqno (15)
\]
The map $(13)$ of Suslin-Voevodsky is the composite of the inverse of (14) with (15).

\bigskip\noindent
We construct the inverse of (14).
Let  a class $[ \Z/m\Z \hookrightarrow E \twoheadrightarrow F]\in \Ext^1_\qfh(F,\Z/m\Z) $ be given. As a morphism in the derived category this class is given by the homomorphism
\[
\begin{tikzpicture}
\matrix (m) [matrix of math nodes, row sep=2em,
column sep=2em, text height=2.5ex, text depth=1.25ex]
{0 & 0 & F & 0 \\
0&  E &  F &0 &. \\
};
\path[->,font=\scriptsize] (m-1-1) edge node[auto]{} (m-1-2);
\path[->,font=\scriptsize] (m-1-2) edge node[auto]{} (m-1-3);
\path[->,font=\scriptsize] (m-1-3) edge node[auto]{} (m-1-4);
\path[->,font=\scriptsize] (m-2-1) edge node[auto]{} (m-2-2);
\path[->,font=\scriptsize] (m-1-1) edge node[auto]{} (m-2-1);
\path[->,font=\scriptsize] (m-1-4) edge node[auto]{} (m-2-4);
\path[->,font=\scriptsize] (m-2-2) edge node[auto]{} (m-2-3);
\path[->,font=\scriptsize] (m-2-3) edge node[auto]{} (m-2-4);
\path[->,font=\scriptsize] (m-1-3) edge node[auto]{$\mathit{id}$} (m-2-3);
\path[->, font=\scriptsize] (m-1-2) edge node[auto]{} (m-2-2);
\end{tikzpicture}
\]
We therefore have to construct a homomorphism $F_1 \longrightarrow E$  making the diagram
\[
\begin{tikzpicture}
\matrix (m) [matrix of math nodes, row sep=2em,
column sep=2em, text height=2.5ex, text depth=1.25ex]
{F_2 & F_1 & F_0 & 0 \\
0&  E &  F &0  \\
};
\path[->,font=\scriptsize] (m-1-1) edge node[auto]{} (m-1-2);
\path[->,font=\scriptsize] (m-1-2) edge node[auto]{} (m-1-3);
\path[->,font=\scriptsize] (m-1-3) edge node[auto]{} (m-1-4);
\path[->,font=\scriptsize] (m-2-1) edge node[auto]{} (m-2-2);
\path[->,font=\scriptsize] (m-2-2) edge node[auto]{} (m-2-3);
\path[->,font=\scriptsize] (m-2-3) edge node[auto]{} (m-2-4);
\path[->,font=\scriptsize] (m-1-3) edge node[auto]{$\mathit{id}$} (m-2-3);
\path[->, font=\scriptsize] (m-1-4) edge node[auto]{} (m-2-4);
\path[->, font=\scriptsize] (m-1-2) edge node[auto]{} (m-2-2);
\path[->, font=\scriptsize] (m-1-2) edge node[auto]{} (m-2-2);
\path[->,font=\scriptsize] (m-1-1) edge node[auto]{} (m-2-1);
\end{tikzpicture}
\]
commutative. The construction is a sheafified version of what we did before.
Let $U\in \Sch/k$ be arbitrary. Consider the diagram

\[
\begin{tikzpicture}
\matrix (m) [matrix of math nodes, row sep=2em,
column sep=2em, text height=2.5ex, text depth=1.25ex]
{0 & \Z/m\Z(U) & E(U\times \Delta^2)& F(U\times \Delta^2) & 0 \\
0&  \Z/m\Z(U) & E(U \times \Delta^1 ) & F(U \times \Delta^1) &0  \\
0&  \Z/m\Z(U) & E(U) & F(U) &0 &. \\
};
\path[->,font=\scriptsize] (m-1-1) edge node[auto]{} (m-1-2);
\path[->,font=\scriptsize] (m-1-2) edge node[auto]{} (m-1-3);
\path[->,font=\scriptsize] (m-1-3) edge node[auto]{} (m-1-4);
\path[->,font=\scriptsize] (m-1-4) edge node[auto]{} (m-1-5);
\path[->,font=\scriptsize] (m-2-1) edge node[auto]{} (m-2-2);
\path[->,font=\scriptsize] (m-2-2) edge node[auto]{} (m-2-3);
\path[->,font=\scriptsize] (m-2-3) edge node[auto]{} (m-2-4);
\path[->,font=\scriptsize] (m-2-4) edge node[auto]{} (m-2-5);
\path[->,font=\scriptsize] (m-3-1) edge node[auto]{} (m-3-2);
\path[->,font=\scriptsize] (m-3-2) edge node[auto]{} (m-3-3);
\path[->,font=\scriptsize] (m-3-3) edge node[auto]{} (m-3-4);
\path[->,font=\scriptsize] (m-3-4) edge node[auto]{} (m-3-5);
\path[->,font=\scriptsize] (m-1-4) edge node[auto]{$\delta^0-\delta^1 +\delta^2$} (m-2-4);
\path[->,font=\scriptsize] (m-1-3) edge node[auto]{$\delta^0-\delta^1 +\delta^2$} (m-2-3);
\path[->,font=\scriptsize] (m-1-2) edge node[auto]{$\mathit{id}$} (m-2-2);
\path[->,font=\scriptsize] (m-2-2) edge node[auto]{$0$} (m-3-2);
\path[->,font=\scriptsize] (m-2-4) edge node[auto]{$0^*-1^*$} (m-3-4);
\path[->,font=\scriptsize] (m-2-3) edge node[auto]{$0^*-1^*$} (m-3-3);
\end{tikzpicture}
\]
Let $\alpha_1\in F(U\times \Delta^1)$ be given. By the smooth base change theorem and since\break $H^1_\et(\Delta^1,\Z/m\Z)=0$, we can lift $\alpha_1$ to $E(U \times \Delta^1)$ after replacing $U$ by a sufficiently fine \'{e}tale cover. Applying $0^*-1^*$ to this lift, we get an element in $E(U)$. This gives the homomorphism $F_1 \to E$. Now let $\alpha_2\in F(U \times \Delta^2)$ be arbitrary.  After replacing $U$ by a sufficiently fine \'{e}tale cover, we can lift $\alpha_2$ to $E( U\times \Delta^2)$. Since $(0^*-1^*)(\delta^0-\delta^1+\delta^2)=0$ this shows that $(\delta^0-\delta^1+\delta^2)(\alpha_2)$ maps to zero in $E(U)$.

This describes the inverse isomorphism to (14). Evaluating at $U=\Spec(k)$ gives back our original construction, hence $(12)$ and $(13)$ are the same maps. This finishes the proof.

\vskip1cm
\small

{\sc  Rikkyo University, Department of Mathematics, 3-34-1 Nishi-Ikebukuro, To\-shima-ku,
Tokyo Japan 171-8501}

\textit{E-mail address:} {\tt geisser@rikkyo.ac.jp}

\medskip
{\sc  Universit\"{a}t Heidelberg, Mathematisches Institut, Im Neuenheimer Feld 288, D-69120 Heidelberg, Deutschland}

\textit{E-mail address:} {\tt schmidt@mathi.uni-heidelberg.de}

\end{document}